\pdfoutput=1
\documentclass[a4paper,english,11pt]{shinyart}
\usepackage[utf8]{inputenc}
\usepackage[T1]{fontenc}
\usepackage{csquotes}
\usepackage{babel}
\usepackage{tikz}
\usepackage{enumitem}
\usepackage{comment}
\usepackage{mathtools}
\usepackage{mdframed}
\usepackage{algorithmic}
\usepackage[ruled,section]{algorithm}
\usepackage{nicecleveref}
\usepackage{subfig}
\usepackage{float}
\usepackage{grffile} % Seems to be required if filenames contain '.'
\usepackage[section]{placeins}
\usepackage{url}
\usepackage{changes-simple}

\mdfdefinestyle{examplebg}{
    backgroundcolor=black!7,%
    hidealllines=true,%
    innertopmargin=-.3em,%
    innerbottommargin=.7em,%
    innerleftmargin=.7em,%
    innerrightmargin=.7em,%
}

\theoremstyle{definition}
\newtheorem{assumption}[definition]{Assumption}

\surroundwithmdframed[style=examplebg]{example}
\makeatletter
\crefname{ALC@unique}{line}{lines}
\makeatother

\newcommand{\field}[1]{\mathbb{#1}}
\newcommand{\N}{\mathbb{N}}

\newcommand{\R}{\field{R}}
\newcommand{\extR}{\overline \R}
\newcommand{\B}{B}

\newcommand{\norm}[1]{\|#1\|}

\newcommand{\inv}[1]{#1^{-1}}
\newcommand{\grad}{\nabla}
\newcommand{\freevar}{\,\boldsymbol\cdot\,}

\newcommand{\Union}\bigcup
\newcommand{\Isect}\bigcap
\newcommand{\union}\cup
\newcommand{\isect}\cap
\newcommand{\bigunion}\bigcup
\newcommand{\bigisect}\bigcap

\newcommand{\defeq}{:=}
\newcommand{\set}[1]{\{#1\}}
\newcommand{\downto}{\searrow}

\newcommand{\subdiff}{\partial}

\DeclareMathOperator*{\argmin}{arg\,min}

\DeclareMathOperator{\interior}{int}

\DeclareMathOperator{\closure}{cl}

\DeclareMathOperator{\Dom}{dom}

\DeclareMathOperator{\lev}{lev}

\DeclareMathOperator{\diag}{diag}

\DeclareMathOperator{\TV}{TV}

\makeatletter
\def \uminus@sym{\setbox0=\hbox{$\cup$}\rlap{\hbox 
        to\wd0{\hss\raise0.5ex\hbox{$\scriptscriptstyle{-}$}\hss}}\box0}
    \def \uminus    {\mathrel{\uminus@sym}}
\makeatother

\newcommand{\iprod}[2]{\langle #1,#2\rangle}

\makeatletter
\def \weaktostar@sym{\setbox0=\hbox{$\rightharpoonup$}\rlap{\hbox 
        to\wd0{\hss\raise1ex\hbox{$\scriptscriptstyle{*\,}$}\hss}}\box0}
    \def \weaktostar    {\mathrel{\weaktostar@sym}}
\makeatother

\def\linear{\mathbb{L}}

\def\extR{\overline \R}

\def\realopt#1{\widehat #1}
\def\this#1{#1^k}
\def\nexxt#1{#1^{k+1}}

\def\realoptx{{\realopt{x}}}
\def\realoptz{{\realopt{z}}}

\def\nextz{\nexxt{z}}

\def\thisz{\this{z}}

\def\taub{t}
\def\sigmab{s}

\DeclareFontFamily{U}{mathx}{\hyphenchar\font45}
\DeclareFontShape{U}{mathx}{m}{n}{<-> mathx10}{}
\DeclareSymbolFont{mathx}{U}{mathx}{m}{n}
\DeclareMathAccent{\widebar}{0}{mathx}{"73}

\def\thisz{\this{z}}
\def\nextz{\nexxt{z}}

\DeclareMathOperator{\dist}{dist}
\DeclareMathOperator{\prox}{prox}

\def\bar{\widebar}

\def\Id{\mathop{\mathrm{Id}}}
\def\levJ{\lev_{J(z^0)}J}

\newcommand{\hz}{{\hat{\sigma}}}
\newcommand{\hx}{{\hat{x}}}

\newcommand{\tx}{{\tilde{x}^k}}

\newcommand{\MM}{C}

\newcommand{\Jk}{{J_k}}
\newcommand{\JK}{{J_K}}

\newcommand{\zk}{{z^{k}}}
\newcommand{\zkn}{{z^{k+1}}}

\newcommand{\tzkn}{{\tilde{z}^{k+1}}}

\newcommand{\dd}{\mathfrak{d}}

\def\txk{\tilde x^k}
\renewcommand{\prox}[3]{\mathrm{prox}_{#1#2} \left(#3 \right) }

\newcommand{\Ball}[2]{\B\left( #1;#2 \right)  }

\newcommand{\figsizeJ}{0.38\textwidth}

\newcommand{\figsizeJc}{0.41\textwidth}

\newcommand{\FloatBarrierA}{\FloatBarrier}

\newcommand{\subfloatrecoii}[4]{\subfloat[#1]{\hspace*{#2}\includegraphics[width=#3]{#4}\hspace*{#2}}}%
\newcommand{\subfloatrecoi}[3]{\subfloatrecoii{#1}{0.18cm}{#2}{#3}}%
\newcommand{\subfloatcolorbar}[4]{\hspace{#1}\vtop{\vskip#2\hbox{\includegraphics[height=#3,keepaspectratio]{#4}}}}%

\def\zset{V}
\def\chiset{\Omega}

\newcommand{\proxGtabaa}{$0$}

\newcommand{\proxGtabac}{$x_i$}
\newcommand{\proxGtabba}{$\delta_\zset(x)$}

\newcommand{\proxGtabbc}{
    $
    \mathrm{proj}_\zset 
    \left(   
    x_i
    \right)
    $
}
\newcommand{\proxGtabca}{$\delta_\zset(x) + \tfrac{\beta}{2} \norm{x - \zk}^2$}

\newcommand{\proxGtabcc}{
    $
    \mathrm{proj}_\zset 
    \left( 
    \frac{\tfrac{1}{\taub}x_i + \beta z_i^k}{\tfrac{1}{\taub} + \beta}
    \right) 
    $
}
\newcommand{\proxGtabda}{
    $ 
    \begin{matrix*}[l] \delta_\zset(x) + \tfrac{\beta}{2} \norm{x - z^k}^2 \\ + B_\mathrm{min}(x) + B_\mathrm{max}(x) \end{matrix*}
    $
}

\newcommand{\proxGtabdc}{
    $
    \left\lbrace \begin{matrix*}[l]
    \text{proj}_\zset \left( \frac{l^2_\mathrm{min} z_\mathrm{min} + \tfrac{1}{\taub}x_i + \beta z^k_i}{l^2_\mathrm{min} + \tfrac{1}{\taub} + \beta}\right) , & x_i < z_\mathrm{min} \\
    
    \text{proj}_\zset \left( \frac{\tfrac{1}{\taub}x_i + \beta z^k}{\tfrac{1}{\taub} + \beta}\right) ,& z_\mathrm{min} \le x_i \le z_\mathrm{max} \\
    
    \text{proj}_\zset \left( \frac{l^2_\mathrm{max} z_\mathrm{max} +  \tfrac{1}{\taub}x_i + \beta z^k_i}{l^2_\mathrm{max} + \tfrac{1}{\taub} + \beta}\right) , & x_i > z_\mathrm{max}
\end{matrix*}\right.
$
}

\def\thetitle{Relaxed Gauss–Newton methods with applications to electrical impedance tomography}
\title{\thetitle}

\author{
    Jyrki Jauhiainen\thanks{Department Of Applied Physics, University of Eastern Finland, Kuopio, Finland. \email{jyrki.jauhiainen@uef.fi}}
    \and
    Petri Kuusela\footnotemark[1]%\thanks{University of Eastern Finland, Kuopio, Finland. \email{petri.kuusela@uef.fi}}
    \and
    Aku Seppänen\footnotemark[1]%\thanks{University of Eastern Finland, Kuopio, Finland. \email{aku.seppanen@uef.fi}}
    \and
    Tuomo Valkonen\thanks{ModeMat, Escuela Politécnica Nacional, Quito, Ecuador \emph{and} Department of Mathematics and Statistics, University of Helsinki, Finland. \email{tuomo.valkonen@iki.fi}}
}

\hypersetup{
  pdftitle={\thetitle},
  pdfauthor={J. Jauhiainen, P. Kuusela, A. Seppänen, and T. Valkonen}
}

\begin{document}
%%%%%%%%%%%%%%%%%%%%%%%%%%%%%%%%%%%%%%%%%%%%%%%

\maketitle

\begin{abstract}
    As second-order methods, Gauss--Newton-type methods can be more effective than first-order methods for the solution of nonsmooth optimization problems with expensive-to-evaluate smooth components. Such methods, however, often do not converge. Motivated by nonlinear inverse problems with nonsmooth regularization, we propose a new Gauss--Newton-type method with inexact relaxed steps. We prove that the method converges to a set of disjoint critical points given that the linearisation of the forward operator for the inverse problem is sufficiently precise. We extensively evaluate the performance of the method on electrical impedance tomography (EIT).
\end{abstract}

%%%%%%%%%%%%%%%%%%%%%%%%%%%%%%%%%%%%%%%%%%%%%%%
\section{Introduction}
\label{sec:intro}
%%%%%%%%%%%%%%%%%%%%%%%%%%%%%%%%%%%%%%%%%%%%%%%
    
The classical Gauss--Newton method can be used for the iterative solution of nonlinear least squares problems $\min_x~ \tfrac{1}{2}\norm{A(x)}^2$. It works by successive linearisation of the nonlinear operator $A \in C^1(\zset; \R^{M})$ defined on $\zset \subset \R^n$. Often, not the least in inverse problems and data science, one wishes to combine such a least squares fitting with a nonsmooth but convex regularization term  $F: \zset \to \R$ incorporating prior information of a good approximate solution to the ill-posed problem $A(x)=0$. We thus wish to solve
\begin{equation}
    \label{eq:minJ}
    \min_x~ J(x) \defeq \frac{1}{2} \norm{A(x)}^2 + F(x).
\end{equation}
One readily extends the idea behind the Gauss--Newton method to this problem: linearise $A$, solve the resulting convex nonsmooth problem to high accuracy, repeat. Unfortunately, such a basic approach rarely converges, especially in inverse problems where $A$ and its differentials almost by definition are not injective. In this work, after several relaxations of the approach, we prove the convergence of a variant of the Gauss--Newton method for \eqref{eq:minJ}, concentrating on applications to electrical impedance tomography (EIT). 

\subsection*{Nonsmooth nonconvex optimization methods}

If $F$ and $A$ are sufficiently smooth, \eqref{eq:minJ} can frequently be solved with Newton's method.
A small degree of nonsmoothness can be dealt with semismooth Newton's method \cite{Mifflin:1977,Qi:1993,Qi:1993a}.
If $F$ is nonsmooth, nonlinear primal-dual proximal splitting (NL-PDPS) \cite{tuomov-nlpdhgm, tuomov-nlpdhgm-redo} is one possibility; see \cite{tuomov-firstorder} for an overview.
Usually NL-PDPS as a first-order method requires thousands of iterations to converge. If the iterations are computationally costly, the method becomes impractical. This can be the case for $A$ the solution operator of a partial differential equation (PDE). We are thus led to Gauss--Newton-type methods that combine both worlds, however, they often fail to converge \cite{tuomov-nlpdhgm}.

Convergence analysis of the classical Gauss--Newton, for the nonlinear least squares problem $\min_x \frac{1}{2}\norm{T(x)}^2$, with $T$ Lipschitz-continuously differentiable, may be found, for example, in \cite{nocedal2006numerical}. In \cite{pang1993nonsmooth} merely locally Lipschitz $T$ is considered. Several works have also studied extensions of the Gauss--Newton method to the general composite minimization problem $\min_x h(T(x))$; see, for example, \cite{burke1995gauss,ferreira2013convergence,li2002convergence}. These works generally assume that the set of minima $C$ of $h$ is ``weakly sharp'', and that the inclusion $T(x) \in C$ has some ``regular points''. In our setting, writing  $h(x, y)=G(x)+F(y)$ for $T(x)=(A(x), x)$, the existence of a ``regular point'' would reduce to the injectivity of the differential $A'(\realoptx)$ at a minimiser $\realoptx$ of $J$. Since, in inverse problems, the range of $A$ is generally much smaller than the domain, such a condition cannot be expected to hold. The assumption of ``weak sharp minima'' amounts to strong metric subregularity of the objective at the solution set. According to \cite{aragon2008characterization}, this is a local form of strong convexity.

In \cite{salzo2012convegence} the Gauss--Newton method is studied for problems of the specific form \eqref{eq:minJ}. There also, $A'(\realoptx)$ has to be injective, and the sub-problem solutions exact. In this case, linear convergence is proved. However, we want to avoid such injectivity assumptions, and also allow the sub-problems to be solved inexactly. To be able to do this, and still obtain convergence, \emph{we will introduce a relaxation term into our subproblems, and relaxation step between the Gauss--Newtons steps}. The former connects our approach to the classical Levenberg--Marquardt method which, indeed, can be seen as a proximal Gauss--Newton method for nonlinear least squares \cite{kaltenbacher2008iterative,hanke1997regularizing}.
We also \emph{will not require the sub-problems to be solved exactly}, merely to obtain sufficient decrease following a condition akin to what has been employed in a different context in \cite{bolte2018nonconvex,attouch2013convergence}. With this, in \cref{sec:gn}, we will show the convergence of iterates of the proposed \emph{Relaxed Inexact Proximal Gauss--Newton} method (RIPGN) to \emph{disjoint components} of critical points. In particular, if the critical points are isolated, we will obtain convergence. 

\subsection*{Electrical impedance tomography}

We will evaluate the proposed method on image (conductivity) reconstruction in Electrical Impedance Tomography (EIT).
This is a large-scale nonlinear PDE-constrained inverse problem. EIT is an imaging technique in which electric conductivity in a target domain is reconstructed from boundary measurements. The relationship between the boundary measurements and the electrical potential and conductivity within the domain are governed by a nonlinear elliptic partial differential equation.
In general, the underlying inverse problem of EIT, which is also known as Calderon's problem \cite{calderon2006inverse}, is ill-posed in the sense that it doesn't depend continuously on the boundary data. However, by assuming certain bounds on the conductivity, it is possible to show an optimal logarithmic modulus of continuity \cite{salo2008calderon}. This, of course, means that even small changes in the conductivity can cause large changes in the boundary values. Cases of nonsmooth conductivities in two dimensions are considered in paper \cite{astala2006calderon}. For cases of piecewise analytic and smooth conductivities in three dimensions, we refer to \cite{kohn1984determining,kohn1985determining} and \cite{sylvester1987global}, respectively.

Theoretical work on the inverse problem of EIT has introduced several direct methods for reconstructing the conductivity. In recent years, so-called D-bar method, which utilizes complex geometrical optics solutions to the Schrödinger formulation of the inverse conductivity problem, has undergone considerable progress \cite{uhlmann2009eit,mueller2012linear}.
In the present, however, we formulate the inverse conductivity problem as a least squares minimization problem between the boundary values from the PDE and measurement data. Optimization and Tikhonov-regularization based approach offers several benefits over the direct methods. It is easier to include physically more accurate boundary conditions, domain shapes and regularization functions. Moreover, in a Bayesian framework, the optimization-based solution can be considered as maximum a posteriori estimates with certain prior distribution \cite{kaipio2006statistical}. With further analysis, error estimates may also be obtained \cite{bardsley2015}.
The underlying optimization problem is, however, often tricky to solve, as the boundary currents depend nonlinearly on the conductivity. This means that the optimization problem is nonconvex. Moreover, total variation type regularization, which help to reconstruct the boundaries of different materials within the target domain, makes the problem nonsmooth.

\subsection*{Organization}

The rest of this paper is organized as follows: first, in \cref{sec:gn}, we examine the convergence of the relaxed inexact proximal Gauss--Newton method. For a certain relaxation parameter, we show that the algorithm converges to a disjoint set of Clarke critical points, given that the  linearisation of the operator $A$ sufficiently well approximates the original operator. In \cref{sec:imp}, we provide a more detailed description of the algorithm and explain how to reliably solve linearised nonsmooth subproblems in the Gauss--Newton scheme. In \cref{sec:app,sec:numexs}, by using EIT as an example, we study numerically and experimentally whether the relaxed Gauss--Newton method improves the computational efficiency of the image reconstructions compared to alternative optimization methods. In these studies, we utilize sythetic data from a water tank setup and experimental measurement data from so-called EIT based sensing skin setup. This is a system for detecting surface changes, eg. cracks, on the given target\cite{Hallaji2014skin}. In \cref{sec:morecases,sec:morecasestwo} we provide further reconstructions for these setups and their variants.

%%%%%%%%%%%%%%%%%%%%%%%%%%%%%%%%%%%%%%%%%%%%%%%
\section{Convergence properties of the relaxed inexact proximal Gauss--Newton method}
\label{sec:gn}
%%%%%%%%%%%%%%%%%%%%%%%%%%%%%%%%%%%%%%%%%%%%%%%

We intend to solve problem \eqref{eq:minJ} by successive linearisations of $A$: for some $\zk$ we take
\[
    A_k(x) \defeq \tilde A_{\zk}(x)
    \quad\text{with}\quad
    \tilde A_y(x) \defeq A(y) + \nabla A(y)^*(x-y)
\]
A standard Gauss--Newton-type approach would then solve on each iteration the linearised, convex problem
\begin{equation}\label{eq:minJk}
    \min_x~ J_k(x) \defeq \frac{1}{2} \norm{A_k(x)}^2 + F(x)
\end{equation}
and update $\zkn \defeq \this{\tilde x}$ to form the linearisation point of the next iteration.
As we have remarked in the introduction, such a method seldom converges. 
Our plan, to obtain a convergent method, is to solve for some \emph{proximal parameter} $\beta>0$ the modified problem
\begin{equation}
    \label{eq:mintJk}
    \min_x~ \tilde{J}_k(x) \defeq \frac{1}{2} \norm{A_k(x)}^2 + F(x) + \frac{\beta}{2} \norm{x-z^k}^2 = J_k(x) + \frac{\beta}{2} \norm{x-z^k}^2.
\end{equation} 
Then we take the linearisation point $\zkn$ as an interpolation between $\this{\tilde x}$ and $\zk$, precisely
\[
    \zkn \defeq (1-w)\zk + w\tx
\]
for a sufficiently small \emph{relaxation parameter} $w \in (0, 1]$. 
Furthermore, we allow $\this{\tilde x}$ to be solved inexactly from \eqref{eq:mintJk}. This yields our outline method of \cref{alg:gn-overrelax}, the \emph{relaxed inexact proximal Gauss--Newton} method (RIPGN).

\begin{algorithm}[t!]
    \caption{Outline of relaxed inexact proximal Gauss--Newton method (RIPGN).}
    \label{alg:gn-overrelax}
    %\algsetup{linenosize=\scriptsize}
    %\footnotesize
    \begin{algorithmic}[1]
        \REQUIRE  Convex, proper, lower semicontinuous $F: \R^N \to \extR$ and $A \in C^1(\Dom F; \R^{M})$.
        \REQUIRE %Proximal parameter $\beta \ge 0$ and 
            Relaxation parameter $w > 0$.
        \STATE Choose an initial iterate $z^0 \in \Dom F$.
        \FORALL{$k \ge 0$}
        \STATE\label{step:gn-overrelax-e} Find an approximate solution $\tx$ to \eqref{eq:mintJk}.
        \STATE Update $\zkn \defeq (1-w)\zk + w\tx$
        \ENDFOR

    \end{algorithmic}
\end{algorithm}

We now prove the convergence of the method with $\beta>0$. In \cref{app:zerobeta} we show that it is possible to take $\beta=0$ under strong metric subregularity. 
We need assumptions that guarantee that the solutions of the linearised subproblems stay in a bounded set, and we need the linearisations $\tilde A_y$ to locally approximate $A$ sufficiently well:

\begin{assumption}
    \label{ass:asecond}     
    $F: \R^N \to \extR$ is convex, proper, and lower semicontinuous, the operator $A \in C^1(\Dom F; \R^{M})$, and
    \[
        J(x) \defeq \frac{1}{2} \norm{A(x)}^2 + F(x).
    \]
    Given an initial iterate $z^0 \in \R^N$, the sublevel set $\levJ$ is bounded, $\inf F > -\infty$, and $A_{\max} \defeq \sup_{z \in \Dom F} \norm{A(z)} < \infty$. Moreover, for some $\dd, \MM>0$ , the linearization error
    \[
        \norm{A(x)- \tilde A_y(x)} \leq \MM \norm{x-y}^2
        \quad (x \in \closure \Ball{y}{\dd},\, y \in \levJ).
    \]
\end{assumption}

Here $B(x,r)$ is the open ball of radius $r$ at $x$ while $\closure B(x, r)$ is its closure. We write $\Dom F \defeq \{ x \in \R^N \mid F(x)<\infty\}$ for the effective domain of $F$ and $\lev_c J \defeq \{ x \in \R^N \mid J(x)\le c\}$ for the $c$-sublevel set of $J$.
We will also write $\subdiff J_k(x)$ for the subdifferential of the convex functions $J_k$ at $x$, and, moreover, denote by $\subdiff_C J(x)$ the Clarke subdifferential of the non-convex function $J$ at $x$, as defined in \cite{clarke1990optimization}. We call a point $x$ satisfying $0 \in \subdiff_C J(x)$ Clarke-critical.
Then we have:

\begin{theorem}
    \label{thm:gn-convergence}
    Suppose \cref{ass:asecond} holds and, for some $\beta, \varepsilon>0$,
    \begin{equation}
        \label{eq:gn-convergence:wbound}
        0 < w \le \min \left\lbrace 1,\frac{\dd}{\sqrt{2\inv\beta (J(z^0) - \inf F)}},\frac{\beta-\varepsilon}{2\MM A_{\max}} \right\rbrace.
    \end{equation}
    On \cref{step:gn-overrelax-e} of \cref{alg:gn-overrelax}, find an approximate minimiser $\this{\tilde x}$ to \eqref{eq:mintJk} specifically satisfying
    \begin{enumerate}[nosep]
        \item For some $e^k \in \subdiff \tilde J(\this{\tilde x})$ we have $\this e \to 0$ as $k \to \infty$, and
        \item either $\tilde J_k(\thisz) \ge \tilde J_k(\this{\tilde x})$ with $\tilde x^k \ne z^k$, or $\tilde x^k=z^k \in \inv{[\subdiff \tilde J_k]}(0)$.
    \end{enumerate}
    Then the iterates satisfy:
    \begin{enumerate}[label=(\roman*),nosep]
        \item\label{item:gn-convergence:decreasing-value}
            $J(\zk)$ is monotonically decreasing; indeed, $J(\zk) \downto L$ for some $L \in \R$.
        \item\label{item:gn-convergence:accumulation}
            Any accumulation point $\hat x$ of $\{\zk\}_{k \in \N}$ is Clarke-critical and satisfies $J(\hat x)=L$;
        \item\label{item:gn-convergence:disjoint-component}
            Indeed, $\dist(z^k, U) \to 0$ for a disjoint component $U$ of $\zset_L \defeq \{\hat x \in \zset \mid 0 \in \subdiff_C J(\hat x),\, J(\hat x)=L\}$.
    \end{enumerate}
\end{theorem}

\begin{proof}
    Suppose first that $\tilde x^k=z^k \in \inv{[\subdiff \tilde J_k]}(0)$ for some $k \in \N$.
    Since $\subdiff \tilde J_k(z^k)=\subdiff J_k(z^k)=\subdiff_C J(z^k)$, we obtain $z^{k+1}=z^k$, so that there is nothing left to prove: the algorithm has converged to a critical point in a finite number of iterations.

    So, by assumption,  $\tilde J_k(\thisz) \ge \tilde J_k(\this{\tilde x})$ with $\tilde x^k \ne z^k$ for all $k \in \N$. 
    Using \eqref{eq:mintJk} we now obtain
    \begin{equation}
        \label{eq:j-first-estim}
        J(z^k) - \Jk (\tx)
        = \tilde J_k(z^k) - \tilde J_k(\tx) + \frac{\beta}{2} \norm{\tx - \zk}^2
        \ge \frac{\beta}{2} \norm{\tx - \zk}^2 > 0.
    \end{equation}
    Since $w \leq 1$, from the convexity of $J_k$ we have
    \begin{equation}
        \label{eq:j-second-estim}
        J(z^k)- \Jk(\zkn)
        \ge J(z^k) - \bigl((1-w)J_k(\zk) + w\Jk(\tx)  \bigr) \\
        = w \bigl( J(\zk)- \Jk(\tx) \bigr).
    \end{equation}
    Consequently, by \eqref{eq:j-first-estim},
    \begin{equation*}
        J(z^k)- \Jk(\zkn) \ge \frac{w\beta}{2} \norm{\tx - \zk}^2 > 0.
    \end{equation*}
    Now we show by induction that 
    \begin{equation}
        \label{eq:induction} 
        J(z^0) \ge J(\zk)
        \quad (k \ge 0).
    \end{equation}
    As a by-product, we will verify \cref{item:gn-convergence:decreasing-value}, and obtain useful estimates for \cref{item:gn-convergence:accumulation} and \cref{item:gn-convergence:disjoint-component}.
    
    \textbf{Induction base:} Obviously $J(z^0) \ge J(\zk)$ holds for $k=0$.
    
    \textbf{Induction step:} Suppose $J(z^0) \ge J(z^k)$. We show $J(z^0) \ge J(z^{k+1})$. From \eqref{eq:j-first-estim} we have
     \[
        J(z^0) - J_k(\txk) \ge J(\zk) - J_k(\txk) \ge \frac{\beta}{2} \norm{\txk - \zk}^2.
    \]
    Since ${J_k}(\txk)\ge \inf F$, we have
    \[
        \norm{\txk - \zk} \le \sqrt{2\inv\beta (J(z^0) - \inf F)} \defeq r,
    \]
    and since $w \le \delta/r$, it follows
    \begin{equation}
        \label{eq:zkn-zk-s}    
        \norm{ \zkn - \zk} = w\norm{\txk- \zk} \le w B \le \frac{\dd}{r}r = \dd,
    \end{equation}
    thus
    $\zkn \in \closure B(\zk; \dd)$. From \cref{ass:asecond} with $h \defeq \zkn - \zk$,
    \begin{equation}
        \label{eq:aest}
        \norm{{A(\zkn)}- {A_k(\zkn)}} \le \MM \norm{\zkn-\zk}^2 \leq \MM \norm{h}^2.
    \end{equation}
   Now using \eqref{eq:aest} and the definition of $A_{\max}$ for the inequality in the next estimate, we obtain
    \begin{equation}
        \label{eq:aest2}
        \begin{aligned}[t]
        \frac{1}{2}\norm{A_k(\zkn)}^2-\frac{1}{2}\norm{A(\zkn)}^2
        &=
        \frac{1}{2}\norm{A(\zkn)-A_k(\zkn)}^2
        \\
        \MoveEqLeft[-1]
        +\iprod{A_k(\zkn)-A(\zkn)}{A(\zkn)}
        \\
        &
        \ge \iprod{A_k(\zkn)-A(\zkn)}{A(\zkn)}
        \ge
        -\MM A_{\max} \norm{h}^2.
        \end{aligned}
    \end{equation}
    Furthermore, using \eqref{eq:aest2},
    \begin{align*}
        %  J(\zk)- J(\zkn) &= J(\zk) - G(\zkn) -F(\zkn) \\
        J(\zk)- J(\zkn) &= J(\zk) - \frac{1}{2}\norm{A(\zkn)}^2 -F(\zkn) \\
        %&\geq J(\zk) - \frac{1}{2}\left( \tM\norm{h}^2 + \norm{ A_k(\zkn)} \right)^2  - F(\zkn)\\
        %&= J(\zk) - \frac{1}{2}\left( \norm{A_k(\zkn)}^2  + \tM^2\norm{h}^4 + 2\tM\norm{h}^2\norm{A_k(\zkn)} \right) - F(\zkn) \\
        &\ge J(\zk) - \frac{1}{2} \norm{A_k(\zkn)}^2 - F(\zkn) - \MM A_{\max} \norm{h}^2 \\
        &= J(\zk) - J_k(\zkn)  - \MM A_{\max} \norm{h}^2.
        \\
    \intertext{Using \eqref{eq:j-second-estim}, \eqref{eq:zkn-zk-s}, and \eqref{eq:j-first-estim}, we continue}
        J(\zk)- J(\zkn)
        &\geq w\bigl( J(\zk)- \JK(\txk) \bigr) - \MM A_{\max} \norm{h}^2 \\
        &= w\bigl( J(\zk)- \JK(\txk) \bigr) - \frac{2w^2 \MM A_{\max} \norm{\txk - \zk}^2}{2} \\
            &\ge w\left( \frac{\beta\norm{\txk - \zk}^2  - 2w \MM A_{\max}\norm{\txk - \zk}^2}{2} \right).
    \end{align*}%
    Since \eqref{eq:gn-convergence:wbound} implies $\beta \ge 2w\MM A_{\max} + \varepsilon$ for some $\varepsilon>0$, we deduce that
    \begin{equation*}
    %    \label{eq:jdecrease}
        J(\zk)- J(\zkn) \ge \frac{w\varepsilon}{2} \norm{\zk - \txk}^2 > 0.
    \end{equation*}
    With this and $J(z^0) \ge J(\zk)$, we get $J(z^0) > J(\zkn)$. This completes the proof of the induction step and consequently \eqref{eq:induction}.
    
    In the process, we obtained
    \begin{equation}
        \label{eq:jdecrease}
        J(\zk)- J(\zkn) \ge \frac{w\varepsilon}{2} \norm{\zk - \tx}^2\quad\text{and}\quad J(\zk) > J(\zkn)
        \quad (k \ge 0).
    \end{equation}    
    Since $\levJ$ is bounded and $J$ is proper and lower semicontinuous, this verifies \cref{item:gn-convergence:decreasing-value}.
    
    To verify \cref{item:gn-convergence:accumulation},
    we observe that summing \eqref{eq:jdecrease} over $\ell=0,\ldots,k-1$ and telescoping gives
    \begin{equation*}
        J(z^0) \ge J(z^k) +  \frac{w\varepsilon}{2} \sum_{\ell=0}^{k-1} \norm{z^\ell - \tilde x^\ell}^2 \ge \inf F + \frac{w\varepsilon}{2}\sum_{\ell=0}^{k-1} \norm{z^\ell - \tilde x^\ell}^2
        \quad (k \ge 1).
    \end{equation*}
    This implies $z^k - \tilde x^k \to 0$.  
    We have assumed that $\this e \in\subdiff \tilde J_k(\this{\tilde x})$ for some $e^k \to 0$.
    With $\subdiff \tilde J_k$ further expanded, using that
    \[
        \grad \left(\frac{1}{2}\norm{A_k(x)}^2\right)=\grad A_k(x) A_k(x)=\grad A(z^k)[A(z^k)+\grad A(z^k)^*(x-z^k)],
    \]
    this is to say
    \begin{equation}
        \label{eq:to-limit-in-subdiff-2}
        \this e \in \grad A(z^k)[A(z^k)+\grad A(z^k)^*(\tilde x^k-z^k)] + \subdiff F(\tilde x^k) + \beta (\tilde x^k-z^k) .
    \end{equation}
    Since $\{\this z\}_{k \in \N} \subset \levJ$, which by assumption is bounded, we can thus find a converging subsequence $z^{k_i} \to \hat x$ for some $\hat x$. Necessarily $\hat x \in \Dom F$.
    
    Recall that the subdifferential mapping $x \mapsto \subdiff F(x)$ is outer semicontinuous \cite{hiriarturruty2001fundamentals}, that is, if $q^{k_i} \in \subdiff F(z^{k_i})$ and also $q^{k_i} \to \hat q$, then $\hat q \in \subdiff F(\hat x)$.
    As $A \in C^1(\Dom F; \R^M)$, passing to the subsequential limit in \eqref{eq:to-limit-in-subdiff-2}, using the outer semicontinuity and $\this e \to 0$, we obtain
    \begin{equation}
        \label{eq:to-limit-in-subdiff-3}
        0 \in \grad A(\hat x)A(\hat x) + \subdiff F(\hat x).
    \end{equation}
    Of course, $\grad A(\hat x)^*A(\hat x)=\grad \left(  \tfrac{1}{2}\norm{A(\hx)}^2 \right) $. %\grad G(\hat x)$.
    By standard calculus rules for the Clarke subdifferential \cite{clarke1990optimization}, \eqref{eq:to-limit-in-subdiff-3} is therefore to say $0 \in \subdiff_C J(\hat x)$.
    This proves \cref{item:gn-convergence:accumulation}.

    Finally, to prove \cref{item:gn-convergence:disjoint-component}, let $\hat x_1$ and $\hat x_2$ be two different accumulation points of $\{z^k\}_{k \in \N}$.
    To reach a contradiction, suppose they would lie in two disjoint subsets $U_1$ and $U_2$ of $\zset_L$. Without loss of generality, we may assume that $\zset_L=U_1 \union U_2$. Since $\zset_L$ is closed (by $J$ being lower semicontinuous and $\subdiff_C J$ outer semicontinuous), so are $U_1$ and $U_2$. We can therefore find $\epsilon>0$ such that $U_1^{2\epsilon}$ and $U_2^{2\epsilon}$ remain disjoint, where $U_j^\epsilon \defeq U_j + \B(0, \epsilon)$, ($j=1,2$). Let $L' \defeq \inf_{x \in \zset \setminus (U_1^\epsilon \union U_2^\epsilon)} J(x)$. Then $L'>L$.
    By definition of $\hat x_1$ and $\hat x_2$ as accumulation points, there exist subsequences $U_1^\epsilon \ni z^{k_i^1} \to \hat x_1$ and $U_2^\epsilon \ni z^{k_i^2} \to \hat x_2$ that satisfy $J(z^{k_i^1}) \to J(\hat x_1) = L < L'$ and $J(z^{k_i^2}) \to J(\hat x_2) = L < L'$. By passing to a subsequence, we may assume without loss of generality that $k_i^1 < k_i^2 < k_{i+1}^1$.
    Since $U_1^{2\epsilon}$ and $U_2^{2\epsilon}$ are disjoint, and $\norm{z^{k+1}-z^k}=w\norm{z^k-\tilde x^k} \to 0$ this implies for $i$ large enough the existence of $k_i^* \in \N$ such that $z^{k_i^*} \in \zset \setminus (U_1^\epsilon \union U_2^\epsilon)$ with $k_i^1 < k_i^* < k_i^2$. Then $J(z^{k_i^*}) \ge L' > L$.
    However, since $\{J(z^k)\}_{k \in \N}$ is decreasing and $J(z^{k_i^1}) \to L$, we also have $\limsup_{i \to \infty} J(z^{k_i^*}) \le L$. This contradiction establishes that $\hat x_1$ and $\hat x_2$ must lie in the same disjoint component of $\zset_L$.
\end{proof}

\begin{remark}[More general data terms]
    Let $g: \R^n \to \R$ be subadditive and $L$-Lipschitz, for example, $g=\norm{\freevar}_p$, $p \in [1, \infty]$.
    How could we replace $\frac{1}{2}\norm{A(x)}^2$ by $g(A(x))$ in \eqref{eq:minJ}? 
    The inequality \eqref{eq:aest2} is the crucial part of the proof to work with such an alternative fitting function.
    Due to subadditivity we have $g(A_k(\zkn))-g(A(\zkn)) \ge -g(A(\zkn)-A_k(\zkn))$.
    If for some $C'>0$ we assume
    \begin{equation}
        \label{eq:gass-alt}
        g(A(\zkn)-A_k(\zkn)) \le C' \norm{h}^2,
    \end{equation}
    then instead of \eqref{eq:aest2} we obtain $g(A_k(\zkn))-g(A(\zkn)) \ge -C' \norm{h}^2$.
    The proof now goes through if we replace the third bound on $w$ in \eqref{thm:gn-convergence} by $\frac{\beta-\epsilon}{2C'}$. For $g=\norm{\freevar}_1$ and $C'=\MM$, \eqref{eq:gass-alt} is simply \cref{ass:asecond}, so no additional assumptions are needed for that choice.
\end{remark}

\begin{remark}[Unique accumulation point under second-order growth conditions]
    \label{rem:second-order-growth}
    If one of the accumulation points $\hat x$ of $\{z^k\}_{k \in \N}$ is actually a unique local minimiser, for example, $J$ satisfies a second-order growth condition around $\hat x$, then $S=\{\hat x\}$ forms a disjoint component of $\zset_L$. Consequently, $\hat x$ has to be the unique accumulation point of $\{z^k\}_{k \in \N}$. It follows that the whole sequence convergences to $\hat x$.
\end{remark}

\begin{remark}[Convergence with a larger relaxation parameter]
    \label{rem:convergence-with-larger-steps}
    There are two obvious strategies to replace the relaxed variable $\zkn$ by $\tzkn \defeq (1-w_k)\zk+w_k\tx$ for some stepwise relaxation parameter $w_k$ that violates the bounds \eqref{eq:gn-convergence:wbound}:
    \begin{enumerate}[label=\alph*)]
        \item Since $\MM A_{\max}$ in the third bound of \eqref{eq:gn-convergence:wbound} arises from \eqref{eq:aest2}, we can replace it by the exact “fractional linearisation error”
        \[
            \max\left\{0, \frac{\norm{A(\tzkn)}^2-\norm{A_k(\tzkn)}^2}{2\norm{\tzkn-\zk}^2}\right\}
            =
            \max\left\{0, \frac{\norm{A(\tzkn)}^2-\norm{A_k(\tzkn)}^2}{2 w_k \norm{\zk-\tx}^2}\right\}.
        \]
        This depends on $w_k$ through $\tzkn$. We therefore need to perform a line search to find (the largest) $w_k$ satisfying this condition subject to the first two bounds of \eqref{eq:aest2}.

        \item If the inequality \eqref{eq:jdecrease} holds for $\tzkn$ in place of $\zkn$. We can again use a line search to find a parameter $w_k \ge w$ satisfying this.
    \end{enumerate}
\end{remark}

%%%%%%%%%%%%%%%%%%%%%%%%%%%%%%%%%%%%%%%%%%%%%%%
\section{Solution of the inner problem and other implementation details}
\label{sec:imp}
%%%%%%%%%%%%%%%%%%%%%%%%%%%%%%%%%%%%%%%%%%%%%%%

In this section, we discuss how to solve the subproblems \eqref{eq:minJk} generated by \cref{alg:gn-overrelax}. Furthermore, we present a framework of how to apply RIPGN to (nonsmooth and nonconvex) regularized nonlinear least squares problems. % Although there is variety of algorithms capable of solving these problems, primal dual proximal splitting (PDPS), which is more commonly known as Chambolle--Pock algorithm \cite{chambolle2010first}, is utilized here, as it has been shown to converge in both theoretical and numerical studies\todo{more citations?}. Furthermore, we will derive a balanced version of the PDPS, in order to handle the changing scale of the linearised forward operator, and possibly speed-up the convergence. The speed-up property of the balancing method will be tested with EIT applications in \cref{sec:app}. In addition, we will describe how to formulate the subproblems for the balanced PDPS. 

%%%%%%%%%%%%%%%%%%%%%%%%%%%%%%%%%%%%%%%%%%%%%%%
\subsection{Balanced primal dual proximal splitting for the linearised subproblem}
%%%%%%%%%%%%%%%%%%%%%%%%%%%%%%%%%%%%%%%%%%%%%%%

To solve the nonsmooth but convex problems \eqref{eq:mintJk}, we utilize a variant of the primal-dual proximal splitting (PDPS) due to  Chambolle and Pock \cite{chambolle2010first}). The basic version of the method applies to $\min G + F_1 \circ K_1$ for some convex $G$ and $F_1$ and a linear operator $K_1$. The function $G$ and the Fenchel conjugate $F_1^*$ need to have easily calculable proximal maps
\[
    \prox{t}{G}{z} \defeq \argmin_x~G(x) + \frac{1}{2t}\norm{x-z},
\]
where $t>0$ is a step length parameter.
However, our problem \eqref{eq:mintJk} with $J_k$ defined in \eqref{eq:minJk} will typically involve several operators; in case of total variation regularization of $x$,
\[
    \min_x~ \frac{1}{2}\norm{A_k(x)}^2 + \alpha\norm{\grad_h x} + \frac{\beta}{2}\norm{x-z^k}.
\]
Proximal maps for functions composed with operators are generally not easily calculable.
Therefore, the linear part of $A_k$ and the discretised gradient $\grad_h$ will both have to go into $K_1$; it will consist of two different blocks with different scales, which moreover vary between the subproblems due to changing linearisations of $A_k$. We will therefore adapt the algorithm to the scales of these blocks following \cite{tuomov-blockcp,pock2011iccv}. 

\subsection{Spatially-adapted primal-dual proximal splitting}

For convex, proper, lower semicontinuous $G: X \to \extR$, $F_1: Y_1 \to \extR$, $F_2: Y_2 \to \extR$ and linear operators $K_1 \in \linear(X; Y_1)$, $K_2 \in \linear(X; Y_2)$, on (finite-dimensional) Hilbert spaces $X,Y_1$, and $Y_2$, we consider
\begin{equation}
    \label{eq:blockproblem}
    \min_{x \in X}~ G(x) + F_1(K_1 x) + F_2(K_2 x).
\end{equation}
With $Kx \defeq (K_1x, K_2x)$ and $y=(y_1, y_2) \in Y \defeq Y_1 \times Y_2$, we can write the problem using the convex conjugates of $F_1$ and $F_2$ as
\[
    \min_{x \in X} \max_{y \in Y}~ G(x) + \iprod{Kx}{y} - F_1^*(y_1) - F_2^*(y_2).
\]
Due to potentially different scales of the ``blocks'' $y_1$ and $y_2$ of $y$, we use two different dual step length parameters for numerical efficiency. This has been called ``diagonal preconditioning'' in \cite{pock2011iccv} and ``spatial adaptation'' in \cite{tuomov-blockcp}. The latter also introduces ways to perform acceleration when strong convexity is present in only some blocks. In either case, without acceleration, such a block-adapted method requires specifying step lengths $\taub, \sigmab_1, \sigmab_2>0$ satisfying
\[
    \Id > \taub \Sigma^{1/2}KK^*\Sigma^{1/2}
    \quad\text{for}\quad \Sigma \defeq \diag(\sigmab_1 \Id, \sigmab_2 \Id),
\]
where we write $\Id: x \mapsto x$ for the identity operator.
Since $KK^*=\begin{psmallmatrix} K_1K_1^* & K_1K_2^* \\ K_2K_1^* & K_2K_2^*\end{psmallmatrix}$, by Young's inequality, this condition holds if for some $\lambda>0$ and estimates $L_1 \ge \norm{K_1}$ and $L_2 \ge \norm{K_2}$,
\begin{equation}
    \label{eq:balancing:stepcond}
    1 > (1+\lambda)\taub\sigmab_1L_1^2
    \quad\text{and}\quad
    1 > (1+\inv\lambda)\taub\sigmab_2L_2^2.
\end{equation}
\Cref{alg:alg-blockcp-fulldual} specializes the spatially adapted or diagonally preconditioned PDPS to the two-dual-block case and these step length conditions; for more general descriptions, stochastic sampling, and acceleration, we refer to \cite{tuomov-blockcp}.
A simple choice to satisfy \eqref{eq:balancing:stepcond} is to take for $\lambda=1$, some $\taub>0$, and small $\delta \in (0, 1)$,
\begin{equation}
    \label{eq:balancing:stepparam}
    \sigmab_1 = (1-\delta)/[2\taub L_1^2]
    \quad\text{and}\quad
    \sigmab_2 = (1-\delta)/[2\taub L_2^2].
\end{equation}
Notice how larger $\norm{K_j}$ will cause correspondingly smaller step length parameter $\sigmab_j$.
This way the method can balance between differing scales of the different blocks of the dual variable.

\begin{algorithm}[t!]
    \caption{Primal-dual proximal splitting with distinct step lengths for two dual blocks}
    \label{alg:alg-blockcp-fulldual}
    \begin{algorithmic}[1]
    \REQUIRE Convex, proper, lower semicontinuous $G: X \to \extR$, $F_1: Y_1 \to \extR$, $F_2: Y_2 \to \extR$ and linear operators $K_1 \in \linear(X; Y_1)$, $K_2 \in \linear(X; Y_2)$.
    \STATE Choose step length parameters $t, s_1, s_2>0$ satisfying \eqref{eq:balancing:stepcond} for some upper bounds $L_1 \ge \norm{K_1}$ and $L_2 \ge \norm{K_2}$ and $\lambda>0$.
    \STATE Choose initial iterates $x^0 \in X$, $y^0_1 \in Y_1$, $y^0_2 \in Y_2$.
    \FORALL{$i \ge 0$ \textbf{until} a stopping criterion is satisfied}
    \STATE $
        x^{i+1} \defeq \prox{\taub}{G}{
            x^i - \taub K_1^* y^i_1 - \taub K_2^* y^i_2
        }
    $
    \STATE
    $
        \bar x^{i+1}
        \defeq
        2x^{i+1}-x^i
    $
    \STATE
    $
        y^{i+1}_1 \defeq \prox{\sigmab_1}{F^*_1}{
            y^i_1 + \sigmab_1 K_1 \bar x^{i+1}
        }
    $
    \STATE
    $
        y^{i+1}_2 \defeq \prox{\sigmab_2}{F^*_2}{
            y^i_2 + \sigmab_2 K_2 \bar x^{i+1}
        }
    $
    \ENDFOR
    \end{algorithmic}
\end{algorithm}

The method has $O(1/N)$ convergence rate for an ergodic gap \cite{tuomov-blockcp}.
Since $F_2^*$ is strongly convex, it would also be possible to update the parameters $t,s_1,s_2>0$ on each iteration to accelerate the method to a mixed $O(1/N^2)+O(1/N)$ convergence rate for $y_2$ \cite{tuomov-blockcp}. %In the interests of simplicity, we have not done this here. However, the

%%%%%%%%%%%%%%%%%%%%%%%%%%%%%%%%%%%%%%%%%%%%%%%
\subsection{Relaxed Inexact Proximal Gauss--Newton}
\label{ssec:RIPGN}
%%%%%%%%%%%%%%%%%%%%%%%%%%%%%%%%%%%%%%%%%%%%%%%

We now explain how we will use \cref{alg:alg-blockcp-fulldual} to solve the sub-problems \eqref{eq:mintJk} for the RIPGN. We now assume that $F$ has the structure $F(x) = F_2(K_2x) = {F}(x) + \delta_V(x)$, $F_2$ is convex, proper and lower semicontinuous, $K_2$ is linear, and $\delta_V$ is the $\{1,\infty\}$-valued indicator function of a set $V \subset \R^N$. We will typically use $V$ to model positivity constraints.
We now formulate \eqref{eq:mintJk}, namely
\[
    \min_x \frac{1}{2}\norm{A_k(x)}^2 + F(x) + \frac{\beta}{2}\norm{x-\this z}^2
\]    
in the form \eqref{eq:blockproblem} by taking
$
    \label{eq:blockproblem:choices}
    F^k_1(y) = \frac{1}{2}\norm{y - b^k}^2
$, 
$
    K^k_1 = \nabla A(z^k)^*
$, and
$b^k = \nabla A(z^k)^*z^k - A(z^k)$
.
Furthermore, we place the proximal and the indicator term into
$
    G^k(x) = \delta_V(x) + \frac{\beta}{2}\norm{x-z^k}^2.
$
We added superscript $k$ to $F_1$, $K_1$, and $G$ to highlight that these terms depend on the outer iteration.
Now the linearised problem \eqref{eq:mintJk} can be written
\begin{equation}\label{eq:Jk}
    \arg\min_x  G^k(x) + F^k_1(K_1^k x) + F_2(K_2x).
\end{equation}
This has the form \eqref{eq:blockproblem} and can be solved with \cref{alg:alg-blockcp-fulldual} using step parameters \eqref{eq:balancing:stepparam}. 

Note that in \cref{thm:gn-convergence} we may consider $\delta_V$ as a part of $F$. However, from computational stand-point, it is usually more efficient to include it into $G$.

The whole process of solving \eqref{eq:minJ}, the relaxed inexact proximal Gauss--Newton method, is described in \cref{alg:gn-overrelax-detailed}. Here we would like to stress that $A(z)$ and $F(z)$ depend on the application. In the next section, we discuss specific choices of these functions in the case of electrical impedance tomography.

\begin{algorithm}[t]
    \caption{Relaxed inexact proximal Gauss--Newton for problem \eqref{eq:minJ}.}
    \label{alg:gn-overrelax-detailed}
    \begin{algorithmic}[1]
        \REQUIRE  Convex, proper, lower semicontinuous $F_2: \R^n \to \extR$, linear and bounded $K_2: \R^N \to \R^n$, convex $\zset \subset \R^N$, and $A \in C^1(\zset; \R^{M})$.
        \REQUIRE $w > 0$, $\delta \in (0,1)$, $t > 0$, and $\beta > 0$.
        \STATE Choose initial iterate $z^0$.
        \STATE $\sigmab_2 \defeq (1-\delta)/[2t \norm{K_2}^2]$
        \FORALL{$k \ge 0$ \textbf{until} a stopping criterion is satisfied}
            \STATE $K^k_1 \defeq \nabla A(z^k)^*$
            \STATE $b^k = \nabla A(z^k)^*z^k - A(z^k)$
            \STATE $\sigmab_1 \;\;\defeq (1-\delta)/[ 2t \norm{K^k_1}^2 ]  $
            \STATE Using \cref{alg:alg-blockcp-fulldual} with parameters $\taub$, $\sigmab_1$, $\sigmab_2$ and initial iterates $x^0 \defeq \zk$, $y_1^0 \defeq 0$, and $y_2^0 \defeq 0$, find an approximate solution $\tx=x^i$ (for large $i$) to \eqref{eq:Jk} 
            \STATE $\zkn \defeq \zk + w (\tx - \zk)$
        \ENDFOR
    \end{algorithmic}
\end{algorithm}

%%%%%%%%%%%%%%%%%%%%%%%%%%%%%%%%%%%%%%%%%%%%%%%
\section{Application to electrical impedance tomography}
\label{sec:app}
%%%%%%%%%%%%%%%%%%%%%%%%%%%%%%%%%%%%%%%%%%%%%%%

We give a brief review of the EIT forward model and its finite element (FE) approximation in a case where measurements consist of electric currents corresponding to a set of potential excitations. We treat the inverse conductivity problem of EIT as a regularized nonlinear least squares problem for which we describe three different regularization schemes.
In this section, as a deviation of the previous section, the unknown of interest is written $\sigma$ instead of $z$ or $x$ to be consistent with typical notation for electrical conductivity. 

\subsection{Forward model of EIT}
\label{ssec:FME}

Due to our measurement equipment, we derive the forward model of EIT in such way that it solves the current through each electrode, given the conductivity within the domain and potential at each electrode. More specifically, in each excitation, one of the electrodes on object's surface is set to a known electric potential, and the rest of the electrodes are connected to ground. Corresponding to each excitation, electric currents through all grounded electrodes are measured.

As the result of the FE approximation, we obtain a nonlinear operator $I(\sigma)$, which together measurement vector $I^m$ and an additional weight matrix $L_A$, forms the data fidelity term $A(\sigma)$ (see below). For details of the FE approximation, we refer to \cite{VossThesis2020}.

Given the electrical conductivity $\sigma$ within domain $\chiset$ and a potential $U^p_{k}$ at each electrode $e_{k}$ during excitation $p$, the forward problem of EIT is to solve the current $I^p_k$ through each electrode. This requires solving also the spatially distributed electric potential $u^p$ inside the domain.
The most accurate physically realizable way to model this is the
Complete Electrode Model (CEM) \cite{cheng1989electrode}. For existence and uniqueness of CEM see \cite{somersalo1992existence}.
With  $\chi = (\chi_1,\chi_2,\chi_3)$ the spatial coordinates within the domain $\Omega \subset \R^3$, CEM is described by a set of equations
\begin{subequations}\label{eq:CEM}
    \begin{alignat}{3}
        \nabla \cdot (\sigma(\chi) \nabla u^p(\chi)) &=0 (\chi\in \chiset),&\quad&
        &u^p(\chi) + \zeta_{k} \sigma \frac{\partial u^p(\chi)}{\partial \hat{n}}
        &= U^p_{k}\quad(\chi\in \partial \chiset_{e_{k}}), \\ %= \bigcup_{k=1}^L    e_{k}  \\
        \int_{\partial \chiset_{e_k}} \sigma \frac{\partial u^p(\chi)}{\partial
            \hat{n}}\: dS &= -I^p_{k},&\text{and}&
        &\sigma\frac{\partial u^p(\chi)}{\partial \hat{n}} &=0, \left(\chi\in \partial \chiset
        \setminus \bigcup_{k=1}^L
            \partial \chiset_{e_k}\right).
    \end{alignat}
\end{subequations}
where $\partial\chiset_{e_{k}}$ is the part of the $\partial\chiset$ covered by $k$'th electrode, $\zeta_k$ is contact impedance, $\hat{n}$ is the outward unit normal of $\chiset$, and $L$ is the number of electrodes. In addition, the currents $I^p_k$ are required to satisfy Kirchhoff's law
$
    \sum_{k=1}^L I^p_{k} = 0
$.
From here on, we assume the contact impedances to be known, $\zeta_k = 10^{-7}$ $\Omega$, as the actual contact impedances in the measurement setups used in this study are negligible.

In order to approximate the solution of the boundary value problem \eqref{eq:CEM} numerically, we utilize Galerkin finite element method (FEM). Following the scheme described in thesis \cite{VossThesis2020}, we write a variational form of the system \eqref{eq:CEM}. Moreover, we use a finite dimensional approximation of the electric potential $u$ as $u^p(\chi) = \sum_{j=1}^{N_u} u^p_j \phi_j(\chi)$ and write the vector of electrode currents for excitation $p$ as $I^p  = \sum_{j=1}^{L-1} \tilde{I}^p_jn_j$ to ensure that the Kirchhoff's current law is fulfilled. Here $\phi_j$ is a basis function for presenting the electric potential, and $n_j,\ j=1,\ldots,L-1$, are vectors that form a basis for the electrode currents. As in a typical Galerkin scheme, $\phi_j$ and $n_j$ are also used as test functions in the variational form. The FE approximation, i.e., the coefficient vector $\theta^p = (u^p_1,\ldots,u^p_N,\tilde{I}^p_1,\ldots,\tilde{I}^p_{L-1})$, is obtained as a solution of the linear system
\begin{equation}\label{eq:CEMfwd} 
    D \theta^p = \tilde{U}^p,\quad\text{where}\quad
    {D} = 
    \begin{pmatrix}
    D_1 & 0\\ D_2 & D_3
    \end{pmatrix} \in \R^{(N+L-1) \times (N+L-1)},
\end{equation}
and the elements of the blocks $D_1$, $D_2$ and $D_3$ are
\begin{align*}
         [D_1]_{ij} &= \int_{\chiset} \sigma(\chi) \nabla \phi_j(\chi) \cdot \nabla \phi_i(\chi) \:
        \mathrm{d}V + \sum_{k=1}^{L} \frac{1}{\zeta_{k}}
        \int_{e_{k}} \phi_j(\chi)\phi_i(\chi)\: \mathrm{d}S,\\
        [D_2]_{kj} &= -\sum_{k=1}^{L} \frac{1}{\zeta_{k}} \int_{e_{k}}
        \phi_j(\chi)(n_k)_{k}\: \mathrm{d}S = -\left( \frac{1}{\zeta_1} \int_{e_1} \phi_j(\chi)\:\mathrm{d}S
        - \frac{1}{\zeta_{k+1}} \int_{e_{k+1}} \phi_j(\chi)\:\mathrm{d}S \right) \\
    [D_3]_{kl} &=  \sum_{k=1}^{L} (n_l)_{k}(n_k)_{k} = 
    \left\{ \begin{array}{cc} 1, & k \neq l \\ 2, & k=l \end{array}
    \right.
\end{align*}
where $i,j=1,\ldots,N$; $j=1,\ldots,N$; and $k,l=1,\ldots,L-1$.
The vector $\tilde{U}^p$ is computed from the known electrode potentials as
\begin{equation} \label{eq:fem_b}
        [\tilde{U}^p]_i = \begin{cases}
        \sum_{k=1}^{L} \frac{U^p_{k}}{\zeta_{k}}
        \int_{e_{k}} \phi_i(\chi)\: \mathrm{d}S, & i=1,\ldots,N\\
        \frac{U^p_{i+1}}{\zeta_{i+1}} \vert e_{i+1}\vert
        -\frac{U^p_{1}}{\zeta_{1}} \vert e_{1} \vert, & i=N+1,\ldots,N+L-1.
               \end{cases}
\end{equation}
Note that the electrode currents $I^p$ are obtained from \eqref{eq:CEMfwd} by first solving the coefficient vector $\theta^p = D(\sigma)^{-1} \tilde{U}^p$ then multiplying $I^p = {\mathcal{K}}\theta^p$ where $\mathcal{K} \in \R^{L \times (N+L-1)}$, ${\mathcal{K}} = [0,\ldots,0,n_1, \ldots n_{L-1}]$. Now the operator $A$ can be written as
\[
    A(\sigma) = {L_A} \left(I(\sigma) - I^m \right),
\]
where $L_A$ arises from the factorization of the inverse noise covariance matrix (precision matrix) $W={L_A}^*{L_A}$ \cite{degroot2005optimal}, $I(\sigma)=(I(\sigma)^1,\ldots,I(\sigma)^L) \in \R^{L^2}$ is a vector containing currents from all excitations, and $I^m$ is the measurement vector corresponding to $I$.
For the linearisation, specifically the components used in \eqref{eq:Jk}, we have
$
    K_1^k = {L_A}\nabla I(\sigma^k)^*$ and $b^k = {L_A} \left(I^m + \nabla_\sigma I(\sigma^k)^*\sigma^k - I(\sigma^k)\right)$.

Finally, we also discretise the conductivity, setting $\sigma = \sum_{i=1}^{N} \sigma_i \varphi_i$, where $\varphi_i$ are linear basis functions. 
Note that $\tilde{U}^p$ is constant with respect to the factors $\sigma_i$, thus the partial derivatives $\tfrac{\partial I^p}{\partial \sigma_i}$ can be solved from
\begin{equation*}
0 = \frac{\partial \tilde{U}^p}{\partial \sigma_i} = \frac{\partial D\theta^p}{\partial \sigma_i} = \frac{\partial D}{\partial \sigma_i}\theta^p + D\frac{\partial \theta^p}{\partial \sigma_i} 
\iff \frac{\partial I^p}{\partial \sigma_i} = \frac{\partial {\mathcal{K}}\theta^p}{\partial \sigma_i} =   -{\mathcal{K}}D^{-1}  \frac{\partial D}{\partial \sigma_i}\theta^p. 
\end{equation*}
For further details on the computation of the Jacobian see \cref{app:jacobian_details}. 
\subsection{Regularization and constraints} %in EIT?
Next we introduce three different regularization schemes for EIT. We utilize these schemes in \cref{sec:numexs}. The first scheme comprises of smoothness-promoting $L^2$-regularization and a barrier function to approximate the positivity constraint. We use this scheme to compare the RIPGN against Newton's method. The other two schemes comprise of total variation (TV) with a positivity constraint, and smoothed TV with the barrier function. The latter is used to compare RIPGN against Newton's method in TV-regularized setting, and the smooth models against nonsmooth models. 
For a detailed description on how to compute the required proximal mappings for \cref{alg:alg-blockcp-fulldual} see \url{http://proximity-operator.net} and \cite{beck2017firstorder}. Additional  mappings are listed in \cref{app:alg_details}.

\subsubsection{Smoothness-promoting regularization with a barrier}
\label{sssec:SPR}

We take the first regulariser
\begin{equation*}
    {F}_\Gamma(\vec \sigma) \defeq \norm{R_\Gamma(\vec \sigma - \vec \sigma_m)}^2,
\end{equation*}
where $\vec \sigma_m$ is the expected value of $\vec \sigma$, and $\vec\sigma=(\sigma_1, \ldots, \sigma_N)$ is the vector of FE factors of $\sigma$. The matrix $R_\Gamma$ is defined by inverse factorization $\left( R_\Gamma^* R_\Gamma \right)^{-1}= \Gamma$ of a Gaussian kernel 
$
    \Gamma_{i,j} = ae^{- \frac{\norm{\chi_i-\chi_j}^2}{2b}}
$ \cite{lipponen2013electrical}.
Furthermore, we introduce a piecewise polynomial barrier function
\begin{equation*}
    {B_\mathrm{min}}(\sigma) \defeq \tfrac{1}{2}\norm{{L_{\mathrm{min}}}(\sigma)(\vec \sigma -  \sigma_\mathrm{min})}^2,
    \quad\text{with}\quad
        [L_\mathrm{min}]_{ij}(\sigma) \defeq 
        \left\lbrace 
        \begin{matrix*}[l]
            l_\mathrm{min},&\text{where } $i=j$ \text{ and } \sigma_i < \sigma_\mathrm{min}\\
            0,&\text{otherwise,}
        \end{matrix*}
        \right.     
\end{equation*}
where $l_\mathrm{min}$ is a coefficient that determines the strength of the barrier function. Now the convex component in \eqref{eq:minJ} is
$
    F(\sigma) = {F}_\Gamma(R_\Gamma \sigma) + {B_\mathrm{min}}(\sigma).
$
As $B_\mathrm{min}$ is diagonal, in the subproblems, it is computationally more efficient to include it into $G^k$. Thus, for formulating the two-block PDPS for the subproblems as in \cref{ssec:RIPGN}, we take
$
    F_2(y) = F_\Gamma(y)$, $K_2\sigma = R_\Gamma\vec \sigma$, and $ G^k(\sigma) = {B_\mathrm{min}}(\sigma) + \delta_\zset(\sigma) + \tfrac{\beta}{2}\norm{\sigma-\sigma^k}^2.
$

\subsubsection{TV regularization and nonsmooth constraints}
\label{sssec:NSR}
In the second scheme we apply nonsmooth total variation regularization with positivity constraints.
Since $\sigma$ is continuous by its finite element construction, its isotropic total variation (TV) \cite{rudin1992nonlinear} can be written as
\begin{equation*}
    \TV(\sigma) = \int_{\chiset} \left| \nabla \sigma(\chi) \right| dV,
\end{equation*}
where $|x| = \sqrt{x_1^2 + x_2^2 + x_3^2}$ is the Euclidean spatial norm.
In linear basis, the spatial gradient of $\sigma$ is constant within an element, meaning $\tfrac{\partial \sigma(\chi)}{\partial \chi_1} = ( \tfrac{\partial \sigma}{\partial \chi_1})_i$ if $\chi$ belongs to element $i$, and the integration yields
\begin{equation*}
    \TV(\sigma) = \sum_{i=1}^{N_E} V_{i} \sqrt{\left( \frac{\partial \sigma}{\partial \chi_1}\right)_i^2 + \left( \frac{\partial \sigma}{\partial \chi_2}\right)_i^2 + \left( \frac{\partial \sigma}{\partial \chi_3}\right)_i^2},
\end{equation*}
where $V_i$ is the volume of the i'th element and $N_E$ is the number of elements in FE basis. This can be expressed
\begin{equation*}
    \TV(\sigma) = \sum_{i=1}^{N_E} \sqrt{\left( R_1\vec \sigma \right)_i^2 + \left( R_2\vec \sigma \right)_i^2 + \left( R_3\vec \sigma \right)_i^2}
    =: \norm{R_\nabla \sigma}_{2,1},
\end{equation*}
where $R_\nabla\sigma \defeq \begin{bsmallmatrix} (R_1 \vec\sigma)^T & (R_2 \vec\sigma)^T & (R_3\vec\sigma)^T\end{bsmallmatrix}^T$ and the components $(i, j)$ of $R_l \in \R^{N_E \times N}$ for $l=1,2,3$ are computed from the basis functions $\varphi_j$ as
\begin{equation*}
    [R_l]_{ij} = \begin{cases}
        V_i \frac{\partial \varphi_j}{\partial\chi_l}, & \varphi_j \text{ when is non-zero in element $i$},\\
        0, & \text{otherwise.}
    \end{cases}
\end{equation*}
For formulating the two-block PDPS for the subproblems as in \cref{ssec:RIPGN}, we now take
$
    F_2(y) = \alpha \norm{y}_{2,1}$, $K_2 = R_\nabla$, and $G^k(\sigma) = \delta_\zset(\sigma) + \tfrac{\beta}{2}\norm{\sigma-\sigma^k}^2.
$

In some examples of \cref{sec:numexs}, we use TV regularization on two-dimensional domains. In those cases, the volume $V_i$ of the element $i$ is replaced by the element surface area and the spatial difference operators, $R_1$ and $R_2$, are computed from the two-dimensional basis functions. Operator $R_3$ is dropped. 

\subsubsection{Smoothed TV regularization and barrier function}
\label{sssec:SNSR}
As the last regularization scheme, we introduce a smoothed version of TV and semismooth barrier functions. The smoothed TV can be written as %$f: \R^N \to \R^{N_E}$
\begin{equation*}
    \tilde{TV}(\sigma)=  \norm{f(\sigma)}_1
    \quad\text{with}\quad
    [f(\sigma)]_i = \sqrt{ (R_1\vec \sigma)_i^2 + (R_2\vec \sigma)_i^2 + (R_3\vec \sigma)_i^2 + \gamma }.
\end{equation*}
Here, $\gamma$ is a smoothing parameter that we set to $\gamma = 10^{-7}$. We also introduce a maximum barrier $B_\text{max}(\sigma)$, by an obvious modification of the minimum barrier $B_\text{min}(\sigma)$ described above. Now the component $F$ in \eqref{eq:minJ} is
$
    F(\sigma) = \alpha \tilde{TV}(\sigma) +  B_\text{min}(\sigma) + B_\text{max}(\sigma),
$
and for the subproblems we have
$
    F_2(y) = \alpha \norm{y}_1$, $K_2(\sigma) = f(\sigma)$, and $G^k(\sigma) = {B_\mathrm{min}}(\sigma) + {B_\mathrm{max}}(\sigma) + \delta_\zset(\sigma) + \tfrac{\beta}{2}\norm{\sigma-\sigma^k}^2.
$
Note that with these notations, the operator $K$ in the subproblem \eqref{eq:mintJk} is nonlinear. 
Hence we solve it using a variant \cref{alg:alg-blockcp-fulldual} for nonlinear $K$ from \cite{tuomov-nlpdhgm,tuomov-nlpdhgm-block}.

%%%%%%%%%%%%%%%%%%%%%%%%%%%%%%%%%%%%%%%%%%%%%%%
\section{Numerical and experimental studies}
%%%%%%%%%%%%%%%%%%%%%%%%%%%%%%%%%%%%%%%%%%%%%%%
\label{sec:numexs}

We evaluate the proposed relaxed inexact proximal Gauss--Newton (RIPGN) method numerically in EIT image reconstruction. In the first set of numerical studies, Cases 1--3 (\cref{ssec:num2d}), we compare RIPGN against Newton's method and NL-PDPS in a circular 2D geometry and in Case 6 (\cref{ssec:num3d}), we demonstrate viability of RIPGN to three-dimensional EIT reconstruction. In Cases 4--5, \cref{ssec:exp2d}, we evaluate the performance of RIPGN with experimental data obtained through EIT-based sensing skin technique. The sensing skin is a surface sensor developed for structural health monitoring: In this technique, the structure is coated with conductive paint and the conductivity of the paint-layer is reconstructed using EIT. If the structure's surface breaks, for example, by cracking, it damages also the paint-layer, and this damage is detected by EIT \cite{Hallaji2014skin}.
We include further experiments in \cref{sec:morecases,sec:morecasestwo}.

\subsection{Computational aspects}
\label{ssec:caspects}
In the numerical studies, we evaluate the convergence of RIPGN (\cref{alg:gn-overrelax-detailed}) with multiple relaxation parameters $w$ and use static values for the parameters $\delta$, $\taub$, and $\beta$. We set $\delta$ to an arbitrary small value $\delta = 0.01$ to satisfy \eqref{eq:balancing:stepcond}, choose $\taub=10^{-6}$ by evaluating the convergence of the first subproblem of Case 3 with multiple step parameters (see \cref{ssec:paramselect}), and set $\beta$ to a small value $\beta = 10^{-10}$; in our experience, $\beta$ has similar impact on the convergence of the \cref{alg:gn-overrelax-detailed} as the relaxation parameter $w$. Every linearised subproblem is solved to 6000 iterations. 

We start the RIPGN, Newton, and NL-PDPS iterations from a homogeneous estimate $\sigma^1$. Furthermore, we introduce minimum and maximum constraints, $\zset_\text{min}$ and $\zset_\text{max}$, by defining the domain $\zset$ as a hypercube $\zset = \left\lbrace \sigma \in \R^N\; :\; \zset_\text{min} \le \sigma_i \le \zset_\text{max},\right.$ $\left. i=1,2,\,\dots,N  \right\rbrace $. \cref{tab:Parameters} shows the parameters that vary between the cases. Note that in this section, we denote the first index as $k=1$ instead of $k=0$.

\input{"Tables/Parameters".tex}

In synthetic tests, Cases 1--3 and 6, we compute the relative error of the estimated conductivity $\hz$ with respect to the true conductivity $\sigma_\text{true}$ as RE$=\norm{\hz-\sigma_\text{true}}/\norm{\sigma_\text{true}} \cdot 100 \%$. Note, however, that due to the simulated measurement noise and the modeling errors caused by the differing mesh sparsities, the true conductivity is often quite far from the actual minimum of the objective function. To highlight this, we compute the objective function at the true conductivity by evaluating the true conductivity at the nodes of the mesh we use in the forward solution. We also compute the relative error of this interpolation, to assess how well the original conductivity could be presented in the forward solution mesh.

We perform all computations in MATLAB 2017b with dual Intel Xeon E5649 @ 2.53/2.93 GHz CPUs and with $99$ GB RAM (1333 Mhz ECC DDR3). We implement crucial components of the construction of the matrix $D$ and the Jacobian $\nabla A$ in C++. We compute the forward solution \eqref{eq:CEMfwd}, the equation $\tilde{\mathcal{K}}\mathcal{D}^{-1}$ for the Jacobian and the linear system for Newton's method through LU decomposition using UMFPACK \cite{Davis:2004:AUV:992200.992206}. In Case 6, we compute the forward solution using BiCGSTAB. 

To catch the stagnation of the RIPGN and Newton's method, we initially stop the iteration if an iterate $z^k$ decreases the value of the objective function less than $0.5$, i.e., if $J(\sigma^k) - J(\sigma^{k-1}) < 0.5$. However, in order to ensure that the iteration does not end prematurely, we compute additional two iterates to check if one of those decreases the objective function by at least $0.5$. If they do, we continue the iteration normally, and if not, we discard these two iterates and take the initial stopping iterate as the estimated solution. We employ this stagnation check after eighth iteration to ensure that at least 10 iterations are computed. For NL-PDPS, we extend these conditions to 700 and 300, respectively. We note that, as in many previous EIT studies \cite{vauhkonen2004image}, line search is used in Newton's method, as the method did not converge within reasonable time with a constant step parameter.

\subsection{Numerical 2D EIT studies}
\label{ssec:num2d}
In Cases 1--3, the geometry of the domain $\Omega$ resembles shallow water tank. The diameter of the tank is 24 cm and the height is 7 cm. Furthermore, the tank has sixteen evenly placed electrodes on the surface; the width and height of the electrodes are 2.5 cm and 7 cm, respectively.

The conductivity inside the tank is constant along the vertical axis, and hence, although the EIT forward model is three-dimensional, the conductivity is two-dimensionally distributed. In the forward model, we map the 2D conductivity to 3D by linear interpolation.

When simulating the measurement data, we present the electrical conductivity in a piecewise linear basis using a tetrahedral mesh consisting 84052 nodes and we approximate the electric potential in a second order polynomial basis consisting 629513 nodes. In the reconstruction, we approximate the 2D conductivity in a piecewise linear basis with  triangular 2D mesh of 1117 nodes; for the forward solver, we map this 2D distribution to piecewise linear 3D distribution (tetrahedral mesh consisting 8189 nodes). Furthermore, we approximate the electric potential with second order polynomial basis functions in a mesh with 56986 nodes.

To simulate actual measurements more realistically, we add Gaussian distributed noise, with std of $0.005\left|I_i\right|$, to each simulated measurement $I_i$.

\subsubsection{Case 1: Smoothness-promoting regularization \& Newton's method}

We first evaluate the RIPGN against Newton's method on a smooth optimization problem. We use the smoothness promoting regularization (Scheme 1; \cref{sssec:SPR}). Furthermore, to match the regularization, the true conductivity is also smooth (\cref{fig:WT_Sim1_Recos}, left): We generate the true conductivity by drawing a sample from a multivariate Gaussian distribution expressing spatial smoothness. This distribution is of the form described in \cref{sssec:SPR}, and its expectation as well as the parameters of the covariance matrix are chosen to be same as in the model used in regularization. We note, however, that since the FE mesh used in inversion is sparser than that in the data simulation, the true conductivity is a not a realization from a model that corresponds to the regularizing function.

\crefname{figure}{Figure}{Figures} \cref{fig:WT_Sim1_J} shows the value of the objective function as a function of iteration number $k$ and computational time $t$ for RIPGN method corresponding to five relaxation parameters $w$ and for the Newton's method. \cref{table:WTSim1} lists the number of iterations required for convergence, value of the objective function at the last iterate, computational time and relative error corresponding to each of these estimates. \cref{fig:WT_Sim1_Recos} illustrates the reconstructed images.
\input{"Figures/Layout/Case1".tex}
\input{"Tables/WT_Sim1".tex}  

\Cref{fig:WT_Sim1_J} and \cref{table:WTSim1} show that in Case 1, Newton's method and RIPGN with $w\leq 3/4$ converge. The reconstructions have small relative errors, as shown by \cref{table:WTSim1}. Smaller relaxation parameters result in increased number of iterations, which in turn increases the computational times, as expected. RIPGN with $w=3/4$ converges in around 7 minutes, while Newton's method converges in about same amount of iterations, but the computation of each iterate is considerably longer, taking around 37 minutes to converge. Hence, although subproblems are solved exactly in Newton’s method, we need the same amount of iterations for convergence as with RIPGN, which solves subproblems inexactly. Longer computational times with Newton's method are mostly due to the line search method.

\Cref{fig:WT_Sim1_Recos} shows that the reconstruction from converging iterations are visually very close to the true conductivity. With step parameters $w = 9/10$ and $w=1$, the RIPGN reconstructions diverge.
Convergence, indeed, cannot be expected for relaxation parameters $w \approx 1$ due to the bound \eqref{eq:gn-convergence:wbound} in \cref{thm:gn-convergence}.

As mentioned in \cref{ssec:caspects}, we also evaluate the objective function at the true conductivity. This gives $J(\sigma_\mathrm{true}) = 1.2057 \cdot 10^5$ and a $0.9294\%$ relative error, meaning that although the true conductivity can be presented quite accurately in the forward solution mesh, the best presentation is very likely far off from the actual minimum of the objective function.  
\FloatBarrierA
\subsubsection{Case 2: Smoothed TV regularization \& comparison with Newton's method}
Because standard Newton’s method cannot be used on non-smooth problems (such as those induced by regularization Scheme 2, \cref{sssec:NSR}), in Case 2, we compare RIPGN to Newton’s method in Scheme 3 (\cref{sssec:SNSR}); a smoothed version of Scheme 2.
In Case 2, the true target contains a circular inclusion of low conductivity ($10^{-3}$ S/m) on a constant background with conductivity of 0.028 S/m. 

\cref{fig:WT_Sim2_STV_J} and \cref{table:WTSim2STV} show that in Case 2, Newton's method takes around 44 minutes to converge while RIPGN with relaxation parameter $w=3/4$ and $w=9/10$ takes around 5--6 minutes. RIPGN diverges again with relaxation parameter $w=1$. The relative errors in Case 2 are larger than in Case 1. This is expected, as the conductivity in Case 1 was a draw from a distribution with statistical properties that corresponded to the regularization that was used. These  errors are further increased as the smooth shapes in Case 1 tend to be more accurately representable with linear interpolation than sharp-edged inclusion in Case 2. The reconstructed images (\cref{fig:WT_Sim2_STV_Recos}) are, however, fairly accurate. Evaluating the objective function at the true conductivity gives $J(\sigma_\mathrm{true}) = 8.6491\cdot 10^4$ with $4.6023\%$ relative error.

\input{"Figures/Layout/Case2".tex}
\input{"Tables/WT_Sim2_STV".tex}
\FloatBarrierA%
\subsubsection{Case 3: TV regularization \& comparison with NL-PDPS} 
\label{sssec:C3}    
In Case 3, we compare RIPGN with NL-PDPS \cite{tuomov-nlpdhgm}. We use the nonsmooth regularization (Scheme 2; \cref{sssec:NSR}).
The target conductivity in Case 3 is the same as in Case 2. 

\Cref{fig:WT_Sim2_TV_Recos} shows no visual differences between the reconstruction computed with RIPGN ($w < 1$) and the reconstruction computed with NL-PDPS. However, \cref{fig:WT_Sim2_TV_J} and \cref{table:WTSim2TV} show that NL-PDPS takes over a week and a half to solve the problem with the desired accuracy, while RIPGN (with $w=3/4$ or $w=9/10$) takes less than 6 minutes. It should be noted though that the total amount of iterations, including the 6000 in each RIPGN linearisation, is considerably fewer with NL-PDPS. This is consistent with earlier studies \cite{tuomov-nlpdhgm,tuomov-nlpdhgm-redo}.

Finally, \cref{fig:WT_Sim2_TV_Recos} and \cref{table:WTSim2TV} show that the unsmoothed total variation slightly improves the reconstruction quality and the relative error from Case 2 (cf. \cref{fig:WT_Sim2_STV_Recos} and \cref{table:WTSim2STV}).

\crefname{figure}{Figure}{Figures}
\input{"Figures/Layout/Case3".tex}
\input{"Tables/WT_Sim2_TV".tex}
\FloatBarrierA
\subsubsection{Effects of the smoothed TV}
\label{ssec:effectsstv}
Next we compare the solutions of the smoothed TV scheme to those of the (nonsmooth) TV scheme. Although the differences between the reconstructions in \cref{fig:WT_Sim2_STV_Recos} and \cref{fig:WT_Sim2_TV_Recos} appear small, closer inspection reveals these to be fundamental. \cref{fig:Sim2_Profiles} shows the true conductivity and three profiles of the true conductivity that are taken along the dashed line. The \cref{fig:Sim2_Profiles} also shows profiles from the solutions computed using Newton's method, RIPGN with smoothed TV and RIPGN with TV. 

\begin{figure}[!t]%
    \centering%
    \includegraphics[height=0.28\textheight]{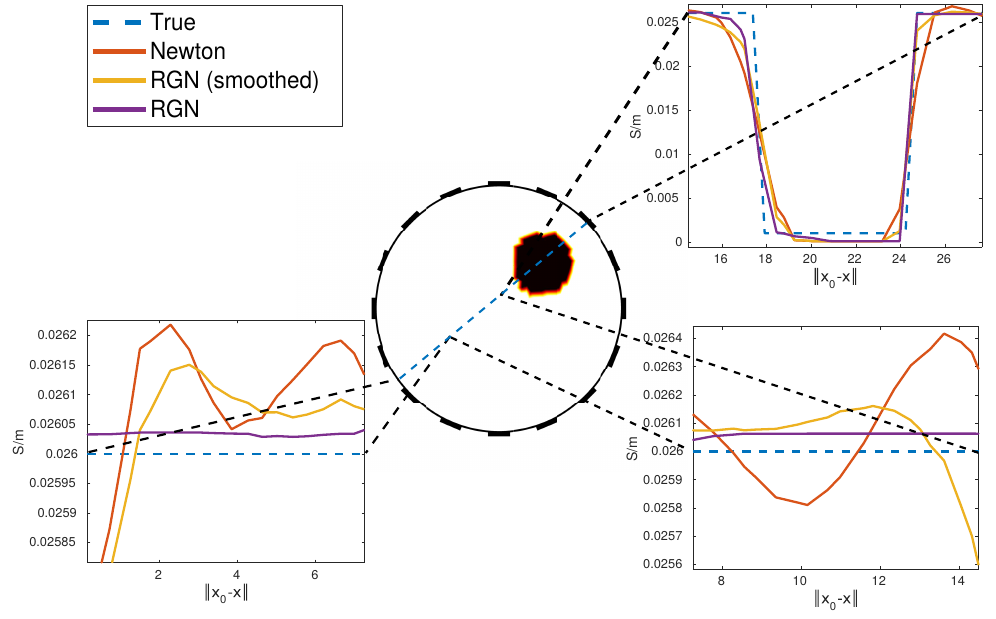}
    \caption{The differences between the smoothed and unsmoothed total variation are distinguishable on closer inspection. Conductivity profile is highlighted with a dashed blue line. Same profile is also taken from smooth Newton and RIPGN reconstructions and nonsmooth RIPGN reconstruction.}%
    \label{fig:Sim2_Profiles}%
\end{figure}%    
\begin{figure}[!tbp]%    	
    \centering%
    \subfloat{\includegraphics[width=\figsizeJ]{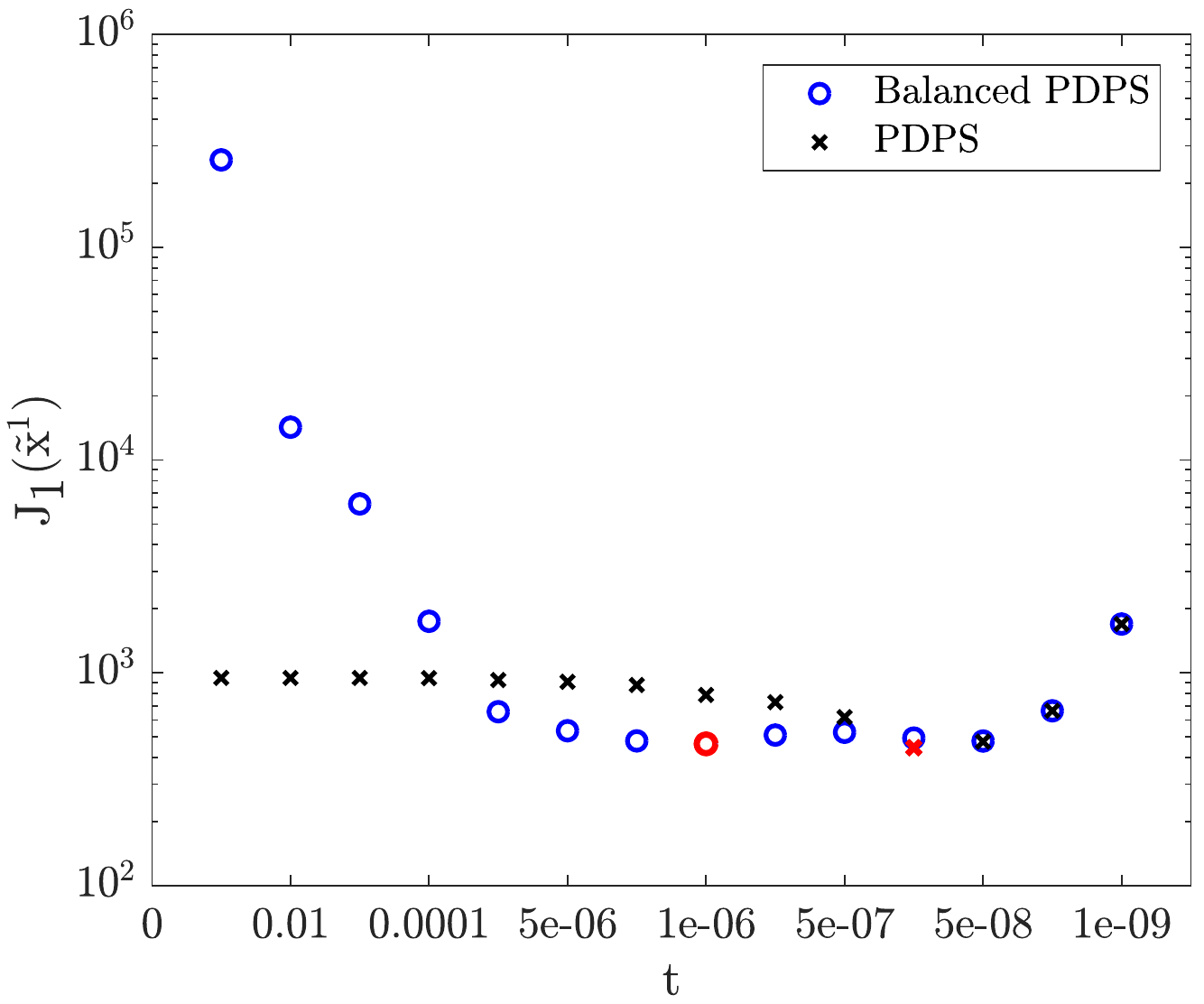} }%
    \subfloat{\includegraphics[width=\figsizeJc]{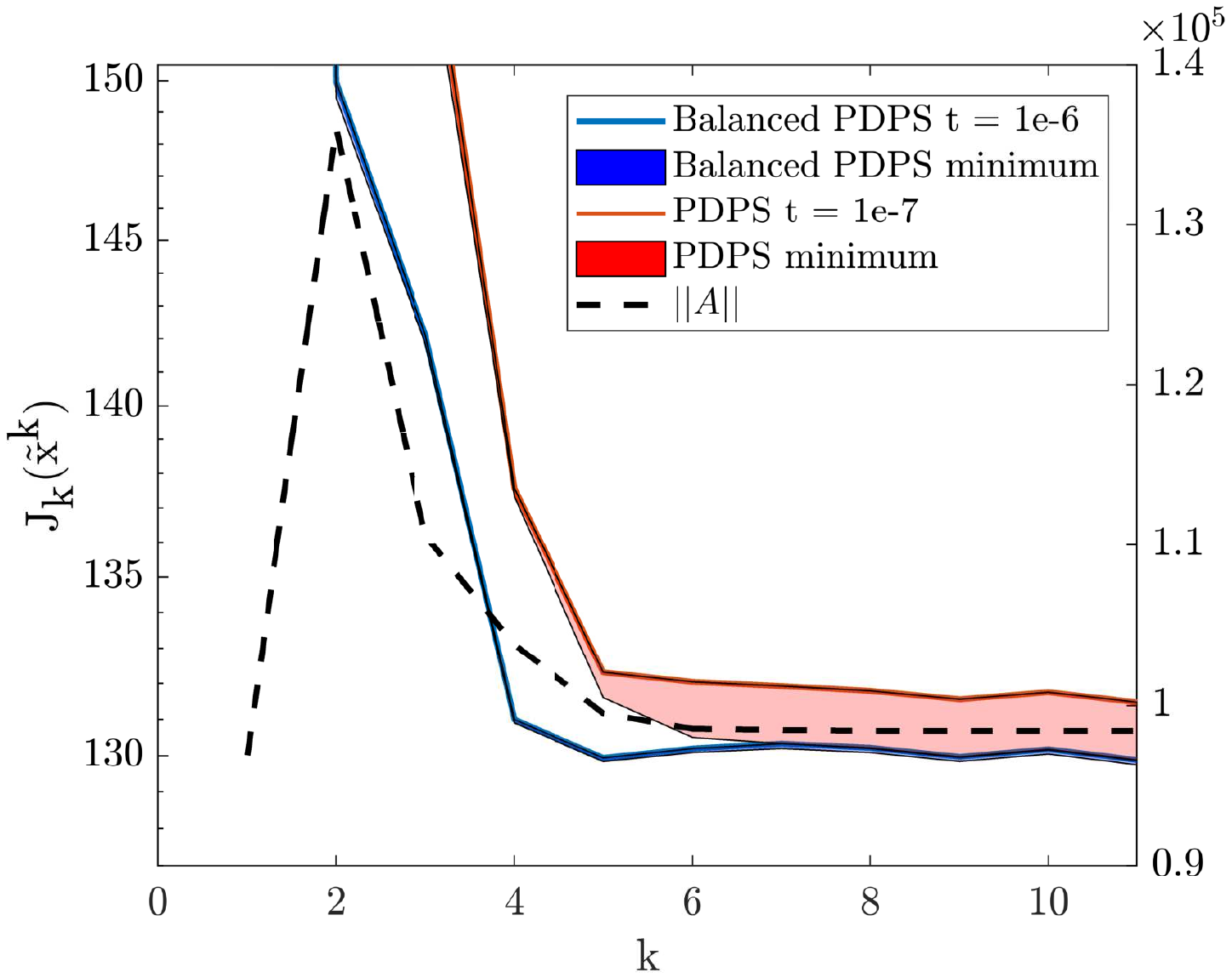}}%
    \caption{Left: Value of the objective function in the first linearised problem at the minimum point estimate $\hx$ as a function of step parameter $t$. The step parameters $t=10^{-6}$ and $t=10^{-7}$ are highlighted in red. Right: Value of the objective function of linearised problem $k$ at the $\hx^k$ with $t=10^{-6}$ for the balanced algorithm and $t=10^{-7}$ non-balanced. Area around the curves highlight the minimum value with any $t$. The dashed line represents the operator norm of $\nabla A$.}%
    \label{fig:BPDPSvsPDPS}%
\end{figure}%    
The profiles in \cref{fig:Sim2_Profiles} illustrate that the solution corresponding to smoothed TV is spatially smoother than that corresponding to non-smoothed TV---the former fails to track the sharp edges in the conductivity. We remind that all solutions are actually piecewise linear due to the choice of basis functions. %There is also a considerable contrast loss at the high conductivity inclusion in all reconstructions. This contrast loss is typical for both TV-regularisation \cite{strong2003edge} and for EIT \cite{karhunen2010electrical}.
\subsubsection{Subproblem parameter selection and balancing}
\label{ssec:paramselect}

In Cases 1--3, we used step parameter $\taub=10^{-6}$ in the linear solver. We chose this step parameter by evaluating the rate of convergence of the first subproblem in Case 3 with multiple step parameters $t$, and then selecting the one that converges fastest.  \cref{fig:BPDPSvsPDPS} (left) shows the value of the objective function at the approximative solution $J_1(\tilde x^1)$ after 6000 iterations. Furthermore, to illustrate the differences between the balanced and the non-balanced method, the figure shows the value of $J_1(\tilde x^1)$ when the problem is solved without balancing, i.e., with $s_1 = s_2 = (tL^2)^{-1}$.

On the right in \cref{fig:BPDPSvsPDPS}, solid lines indicate the value of $J_k(\tilde x^k)$ when the problem is solved using both the balanced and the non-balanced methods with step parameters $\taub = 10^{-6}$ and $\taub = 10^{-7}$ respectively. Areas below the curves show the minimum with any of the tested parameters. For this experiment, the outer iteration is advanced with relaxation parameter $w=3/4$ using solutions from the balanced method with $\taub = 10^{-6}$. For the curiosity, the operator norm of $\nabla A$ is also shown in the figure.

\cref{fig:BPDPSvsPDPS} shows that although both methods converge almost equally in the first subproblem, the balanced method outperforms normal PDPS in the subsequent problems. Furthermore, \cref{fig:BPDPSvsPDPS} shows that 
unlike with the non-balanced PDPS, in the balanced PDPS, the optimal step parameter remains almost unchanged at every linearisation.\FloatBarrierA%
\subsection{Experimental studies}
\label{ssec:exp2d}
The measurement device we use in the experimental studies is manufactured by Rocsole Ltd. (\url{www.rocsole.com}). This device utilizes a typical ECT measurement principle in which each electrode is sequentially set to a known sinusoidal potential, while the others remain grounded. The currents induced by the potential differences are then sampled, in this case with $1$ MHz sampling frequency, and the amplitude of the induced current is computed from the samples using discrete Fourier transform. Here the excitation frequency is set to $39$ kHz and measurements used in the reconstruction are time averages of the computed amplitudes over one minute time period. 

\subsubsection{Cases 4-5: Sensing skin \& crack detection}

In Case 4, we test RIPGN in a crack detection problem arising from EIT-based sensing skins (see \cite{Hallaji2014skin}). 
Computationally this crack detection problem differs from the inclusion detection in a typical water tank geometry, because cracks cause sharp but spatially narrow inclusions of low conductivity on the background conductivity of the paint layer. Furthermore, the conductive paint is far from being homogeneous in thickness and consequently, the background conductivity is inhomogeneous. To take into account this inhomogeneity we follow an approximative data correction approach proposed in \cite{Hallaji2014skin}. In addition, we exploit the fact that the cracks never increase the conductivity, allowing us to constrain the conductivity from above. 

The sensing skin used in the study is painted with Kontakt Chemie EMI 35 conductive graphite paint onto a rectangular plexiglass. The side lengths of the plexiglass are 44 cm and 42 cm and each side has seven 2.5 cm $\times$ 1.25 cm electrodes. Furthermore, four 2.5 cm $\times$ 2.5 cm electrodes are placed in the middle of the sensing skin.

From the sensing skin measurements, we compute a smoothed TV solution with Newton's method and RIPGN (Case 4), and a nonsmooth TV solutions with RIPGN (Case 5). The triangular mesh used in the computations has 3147 nodes for the conductivity represented in a piecewise linear basis and 12281 nodes for the electric potential in second order basis. Parameters used in these cases are shown in \cref{tab:Parameters}.

\input{"Figures/Layout/Case4".tex}
\input{"Tables/Skin_Meas1_STV".tex}

\cref{fig:Skin_Meas1_STV_Recos} (left) shows a photograph of the sensing skin in Case 4. The crack in the photograph is highlighted in red as the crack is very narrow. 

\cref{fig:Skin_Meas1_STV_J} shows that for every relaxation parameter RIPGN converges considerably better than with $w=1$ in Cases 1--3. However, the value of the objective function oscillates slightly over the last few iterations when $w> 1/2$. The better convergence with relaxation parameter $w\le 1/2$ is also confirmed by \cref{table:SkinMeas1STV}. The objective function with Newton's method converges to similar values as RIPGN with the larger step parameters. Note that in this case, the iteration time with Newton's method is considerably shorter than in Cases 1--3 due to the two-dimensional forward model. Furthermore, \cref{fig:Skin_Meas1_STV_Recos} shows that the reconstructed images capture the shape and length of the crack rather well. In this example, the effect of relaxation parameter to the quality of RIPGN-based reconstruction is very small, and even the difference between the RIPGN- and Newton- based reconstructions is somewhat negligible. We note, again, that the choices of the optimization method and relaxation parameter do have an effect on the converge and computation speed, as shown by \cref{table:SkinMeas1STV}.

In Case 5, the sensing skin dataset used in Cases 4 is used to reconstruct TV regularized solution (Scheme 2) with RIPGN. The results are shown in\crefname{figure}{Figures}{Figures} \cref{fig:Skin_Meas1_TV_J}--\ref{fig:Skin_Meas1_TV_Recos} and \cref{table:SkinMeas1TV}.
\crefname{figure}{Figure}{Figures}
Comparing these results with results in Case 4 shows that the contrast between the crack and the background conductivities is higher when the non-smooth model is used (Scheme 2). Again, the computational times are shorter than in the smoothed case (see \cref{sssec:C3}). Apart from these differences, the results are fairly similar to smoothed TV.    
\input{"Figures/Layout/Case5".tex}
\input{"Tables/Skin_Meas1_TV".tex}

\subsection{Numerical 3D EIT study}
\label{ssec:num3d}
In Case 6, we evaluate the feasibility of RIPGN to three-dimensional EIT. The geometry resembles a cylinder that has a radius of 14 cm and a height of 26 cm. Furthermore, the cylinder has four horizontal layers of electrodes on the surface. Each layer contains 10 evenly placed square electrodes with side length of 3 cm. The gap between each electrode layer is 4 cm. The cylinder contains a resistive inclusion with conductivity of $10^{-3}$ S/m on a background conductivity of 0.028 S/m.

In the data simulation we present the electrical conductivity in a piecewise linear basis with 210860 nodes, and the electric potential in a second order polynomial basis with 1632276 nodes. Furthermore, the inversion mesh has 18835 nodes for the conductivity and 135504 nodes for the potential. The reconstructions are computed with Scheme 3 (\cref{sssec:NSR}). 

\cref{fig:WT3D_Meas1_View1_Recos} shows that in Case 6 the relaxation parameter has negligible impact on the reconstruction quality and the reconstructions look very similar to the true conductivity distribution. \cref{fig:WT3D_Sim1_J} and \cref{table:WT3DSim1} show that, even in terms of the final value of the objective function, RIPGN converges similarly with every step parameter. Clearly, in this case we get no benefits for lowering the step parameter as lowering it only increases the amount of iterations required to satisfy the convergence criteria; with step parameter $w=1/4$ it takes 47 iterations, while with $w=1$ it takes only 9. This is also reflected in the computational times. Furthermore, these computational times are considerably longer compared to the previous cases as number of nodes, elements, and electrodes in the model are greater.  
As in the previous synthetic cases, the true conductivity is known and evaluating the objective function at $\sigma_\mathrm{true}$ yields $J(\sigma_\mathrm{true}) = 7.0220 \cdot 10^4$. Furthermore, the relative error is RE $= 1.250$\%.
\input{"Figures/Layout/Case6".tex}
\input{"Tables/WT3D_Sim1".tex}

\section{Conclusions}

We proposed a novel relaxed inexact proximal Gauss–Newton (RIPGN) method, and studied it both theoretically and numerically. We applied the method to image reconstruction from electrical impedance tomography (EIT) measurements which is a large-scale non-linear inverse problem governed by a PDE model.

We showed that the RIPGN converges to a disjoint set of Clarke critical points under conditions that hold for typical inverse problems. Furthermore, we presented a framework for the application of RIPGN to such problems. We confirmed the efficacy of the RIPGN on synthetic and experimental EIT data. These studies showed that by adjusting the relaxation parameter $w$, the iterates generated by the RIPGN converge to solutions that meaningful for EIT applications. Furthermore, when $w$ was appropriately selected, the RIPGN estimates were significantly faster to compute than more conventional estimates produced by Newton's method in the smooth case and the NL-PDPS in the nonsmooth case.

Overall, RIPGN combined with (NL-)PDPS offers a flexible framework to solve various nonconvex and nonsmooth problems. In EIT, the greatest advantage of the method was achieved with nonsmooth TV regularization. Following the implementation of this work, RIPGN method can be straightforwardly adopted also to a variety of other  optimization problems---those associated with other non-smooth regularization schemes as well as other imaging/reconstruction applications yielding non-convex optimization problems. In the future, this may enable handling such large-scale problems without need for smoothing and/or reducing the model complexity, which both can lead to loss of contrast and appearance of imaging artefacts. Moreover, the RIPGN might even enable---via computational speed-up---realizations of high-contrast real-time imaging in some applications.

%%%%%%%%%%%%%%%%%%%%%%%%%%%%%%%%%%%%%%%%%%%%%%%
\section*{Acknowledgments}
%%%%%%%%%%%%%%%%%%%%%%%%%%%%%%%%%%%%%%%%%%%%%%%

This project has received funding from the European Union`s Horizon 2020 research and innovation programme under grant agreement No 764810. The research was also funded by the Academy of Finland (Centre of Excellence of Inverse Modelling and Imaging, 2018-2025, project 303801).

T.~Valkonen has been supported by Academy of Finland grants 314701 and 320022 as well as Escuela Politécnica Nacional internal grant PIJ-18-03.
    
\appendix
%%%%%%%%%%%%%%%%%%%%%%%%%%%%%%%%%%%%%%%%%%%%%%%
\section{Geometric justification for zero proximal parameter}
\label{app:zerobeta}
%%%%%%%%%%%%%%%%%%%%%%%%%%%%%%%%%%%%%%%%%%%%%%%

We now improve \cref{thm:gn-convergence} by showing that we can take the proximal parameter $\beta=0$ provided $\this e$ is small enough and a critical point satisfies certain geometric conditions. We will then also obtain local convergence to this specific critical point.
The rough plan of work is to show that \eqref{eq:j-first-estim} holds under these conditions for some $\beta>0$ despite the algorithm employing $\beta=0$.
Throughout, we take $J$ as in \eqref{eq:minJ} and for brevity write
\[
    G(x) \defeq \frac{1}{2}\norm{A(x)}^2
    \quad\text{and}\quad
    G_k(x) \defeq \frac{1}{2}\norm{A_k(x)}^2 = \frac{1}{2}\norm{A(z^k) + \grad A(z^k)(x-z^k)}^2.
\]

We will for some $\rho>0$ on \cref{step:gn-overrelax-e} of \cref{alg:gn-overrelax}, 
\begin{equation}
    \label{eq:beta=0:alg-mod}
     \text{solve \eqref{eq:minJk} for $\this{\tilde x}$ to such accuracy that $\norm{\this e} \le \rho \norm{\this{\tilde x}-\this z}$ for some $e^k \in \subdiff J_k(\this{\tilde x})$.}
\end{equation}

\begin{lemma}
    \label{lemma:beta=0}
    Suppose \cref{ass:asecond} holds. In \cref{alg:gn-overrelax} use \eqref{eq:beta=0:alg-mod}.
    If $\this q \defeq \this e -\grad G_k(\this{\tilde x}) \in \subdiff F(\this{\tilde x})$satisfies
    \[
        F(\thisz)-F(\this{\tilde x}) \ge \iprod{\this q}{\thisz-\this{\tilde x}} + \frac{1}{2}\norm{z-\this{\tilde x}}_{\Gamma_k}^2
    \]
    for some operator $\Gamma_k$ such that $\grad A(\this z)\grad A(\this z)^*+\Gamma_k \ge (2\rho+\beta) I$ for some $\beta>0$, then \eqref{eq:j-first-estim} holds.
    If \eqref{eq:gn-convergence:wbound} holds for this $\beta$, then the conclusions of \cref{thm:gn-convergence} hold.
\end{lemma}

\begin{proof}
    We have $\this q=\this e - \grad A_k(\this{\tilde x})A_k(\this{\tilde x}) = \this e - \grad A(\thisz)[A(\thisz)+\grad A(\thisz)^*(\this{\tilde x}-\thisz)]$.
    Since we take $\beta=0$ in the algorithm, $\this{e} \in \subdiff J_k(\this{\tilde x})$.
    Therefore
    \begin{equation*}
        \begin{split}
        J(\this z)-J_k(\this{\tilde x})
        &
        =
        \frac{1}{2}\norm{A(\this z)}^2
        -\frac{1}{2}\norm{A_k(\this{\tilde x})}^2
        +F(\this z) - F(\this{\tilde x})
        \\
        &
        \ge
        \frac{1}{2}\norm{A(\this z)}^2
        -\frac{1}{2}\norm{A_k(\this{\tilde x})}^2
        +\iprod{\this q}{\this z-\this{\tilde x}}
        +\frac{1}{2}\norm{\this z-\this{\tilde x}}_{\Gamma_k}^2
        \\
        &
        =
        \frac{1}{2}\norm{A(\this z)}^2
        -\frac{1}{2}\norm{A_k(\this{\tilde x})}^2
        -\iprod{A_k(\this{\tilde x})}{\grad A_k(\this{\tilde x})^*(\this z-\this{\tilde x})}
        \\
        \MoveEqLeft[-1]
        +\iprod{\this e}{\this z-\this{\tilde x}}
        +\frac{1}{2}\norm{\this z-\this{\tilde x}}_{\Gamma_k}^2.
        \end{split}
    \end{equation*}
    We expand and simplify
    \begin{equation*}
        \begin{split}
        \frac{1}{2}\norm{A(\this z)}^2
        &
        -\frac{1}{2}\norm{A_k(\this{\tilde x})}^2
        -\iprod{A_k(\this{\tilde x})}{\grad A_k(\this{\tilde x})^*(\this z-\this{\tilde x})}
        \\
        &
        =\frac{1}{2}\norm{A(\this z)}^2
        -\frac{1}{2}\norm{A(\this z)+\grad A(\this z)^*(\this{\tilde x}-\this z)}^2
        \\
        \MoveEqLeft[-1]
        -\iprod{A(\thisz)+\grad A(\this z)^*(\this{\tilde x}-\this z)}{\grad A(\this z)^*(\this z-\this{\tilde x})}
        \\
        &
        =
        \frac{1}{2}\norm{\grad A(\this z)^*(\this{\tilde x}-\this z)}^2.
        \end{split}
    \end{equation*}
    Using the assumption $\norm{\this e} \le \rho \norm{\this{\tilde x}-\this z}$ thus
    \begin{equation*}
        \begin{split}
        J(\this z)-J_k(\this{\tilde x})
        % &
        % \ge
        % \frac{1}{2}\norm{\this z-\this{\tilde x}}_{\grad A(\this z)^*\grad A(\this z)+\Gamma_k}^2+\iprod{\this e}{\this z-\this{\tilde x}}
        % \\
        &
        \ge
        \frac{1}{2}\norm{\this z-\this{\tilde x}}_{\grad A(\this z)\grad A(\this z)^*+\Gamma_k}^2 -\rho \norm{\this z-\this{\tilde x}}^2.
        \end{split}
    \end{equation*}
    This and the assumption $\grad A(\this z)\grad A(\this z)^*+\Gamma_k \ge (2\rho+\beta) \Id$ prove \eqref{eq:j-first-estim}.
\end{proof}

We now merely assume the conditions of the lemma in the limit:

\begin{theorem}
    \label{theorem:beta=0}
    Suppose $\realopt q \defeq - \grad G(\realoptx) \in \subdiff F(\realoptx)$ satisfies $F(z)-F(\realoptx) \ge \iprod{\realopt q}{z-\realoptx} + \frac{1}{2}\norm{z-\realoptx}_{\Gamma}^2$ for all $z$ and some operator $\Gamma$ such that $\grad A(\realoptx)\grad A(\realoptx)^*+\Gamma \ge (2\rho+\theta) \Id$ for some $\theta,\rho>0$.
    Take any $\beta \in (0, \theta)$ satisfying \eqref{eq:gn-convergence:wbound} and initialize $z^0$ close enough to $\realoptx$.
    In \cref{alg:gn-overrelax} use \eqref{eq:beta=0:alg-mod}.
    Then the conclusions of \cref{thm:gn-convergence} hold.
\end{theorem}

\begin{proof}
    Let $\this q \defeq \this e -\grad G_k(\tilde x^k) \in \subdiff F(\this{\tilde x})$.
    By the outer semicontinuity of the convex subdifferential $\subdiff F$ \cite{hiriarturruty2001fundamentals}, and the continuity of $\grad A$ and $A$, it is clear that for all $\epsilon>0$ that there exists $r'>0$ such that $\norm{\this{\tilde x}-\realoptx},\norm{\this z-\realoptx} \le r'$ ensures $\norm{\this q-\realopt q} \le \epsilon$, $\grad A(\this z)\grad A(\this z)^*+\Gamma \ge (2\rho+\beta) \Id$, and $F(\this z)-F(\this{\tilde x}) \ge \iprod{\this q}{\this z-\this{\tilde x}} + \frac{1}{2}\norm{\this z-\this{\tilde x}}_{\Gamma}^2$.
    Therefore, if we can ensure that $\{\thisz\}_{k \in \N}, \{\this{\tilde x}\}_{k \in \N}  \subset \B(\realoptx, r')$ for some small enough $r'>0$, the claim follows from \cref{lemma:beta=0}.

    Since $\this{\tilde x}=\inv w(\nextz-\thisz)+\thisz$, it suffices to show for some small $r>0$, for all $k \in \N$, that $\thisz \in \B(\realoptx, r)$, and that $\norm{\thisz-z^{k-1}} \le r$. We moreover claim that $J(\thisz) \le J(\realoptx)+ \delta r^2\varepsilon/(2w)$ for some $\delta \in (0, 1]$. We prove all of this by induction. The induction basis follows from initializing $z^0=z^{-1}$ close enough to $\realoptx$, that is, with $r>0$ small enough.
    For the induction step, assume the claim holds for $k$. We will prove that it holds for $k+1$. Indeed, by \cref{lemma:beta=0}, \eqref{eq:j-first-estim} holds for $k$. Thus, by the proof \cref{thm:gn-convergence}, \eqref{eq:jdecrease} holds for $k$:
    $
        J(\zk)- J(\zkn) > \frac{w\varepsilon}{2} \norm{\zk - \tx}^2.
    $
    By the inductive assumption and $J(\nexxt z) \ge J(\realoptx)$, thus
    \[
        \frac{\varepsilon}{2w}\norm{\nextz-\thisz}^2 \le J(\thisz) - J(\nextz) \le J(\thisz)-J(\realoptx) \le \frac{\delta r^2\varepsilon}{2w}.
    \]
    This shows $\norm{\nextz-\thisz} \le r$.
    Since $J(\nexxt z) \le J(\this z)$, also $J(\nexxt z) \le J(\realoptx)+ \delta r^2\varepsilon/(2w)$.

    It remains to prove $\nexxt z \in \B(\realoptx, r)$.
    We have $\realopt q= - \grad A(\realoptx)A(\realoptx)$ and for $z \in \B(\realoptx, r'')$ with $r''$ small enough, $A(\realoptx)=A(\realoptz)+\grad A(\realoptx)(\realoptx-z)+O(\norm{z-\realoptx}^2)$.
    Therefore, arguing similarly to \cref{lemma:beta=0},
    \[
        J(z)-J(\realoptx)
        \ge \frac{1}{2}\norm{z-\realoptx}^2_{\grad A(\realoptx)\grad A(\realoptx)^*+\Gamma}
        - O(\norm{z-\realoptx}^2)
        \ge c \norm{z-\realoptx}^2
    \]
    for any $0<c<\theta+2\rho$ and $z \in \B(\realoptx, r'')$.
    Since $\thisz \in \B(\realoptx, r)$ and, as we have shown, $\norm{\nextz-\thisz} \le r$, we have $\nextz \in \B(\realoptx, 2r)$. Therefore, taking $r < r''/2$, we have $\nextz \in \B(\realoptx, r'')$.
    Taking $z=\nextz$, it now follows
    \[
        \frac{\delta r^2\varepsilon}{2w} \ge J(\thisz) - J(\realoptx) \ge J(\nextz)-J(\realoptx)
        \ge c \norm{\nextz-\realoptx}^2.
    \]
    Therefore, if $\delta>0$ is small enough, $\nextz \in \B(\realoptx, r)$.
    This finishes the induction and the proof.
\end{proof}

We now need to obtain some local strong convexity of $F$. We concentrate on total variation; in the EIT problems that we consider in \cref{sec:app}, more local strong convexity could be obtained from the box constraints.
Related geometric approaches in \cite{tuomov-subreg,lewis2002active,lewis2013partial,liang2014local,garrigos2017convergence} show the local linear convergence of convex optimization methods, and even globally to submanifolds. The next lemma establishes the fundamental idea of the approach. The condition in it has been related to the \emph{strong (metric) subregularity} of the subdifferentials $\subdiff F$ \cite{aragon2008characterization}.

\begin{lemma}
    \label{lemma:interior-local-sc}
    Let $F: \R^n \to \extR$ be convex and $q \in \interior \subdiff F(x)$ for some $x \in \R^n$.
    Then for any $\gamma>0$, for some $\rho>0$, $F(z)-F(x) \ge \iprod{q}{z-x} + \frac{\gamma}{2}\norm{z-x}^2$ for all $z \in \B(x, \rho)$.
\end{lemma}

\begin{proof}
    By the definition of the convex subdifferential,
    \[
        F(z) - F(x) \ge \sup_{q' \in \subdiff F(x)} \iprod{q'}{z-x}
        = \iprod{q}{z-x} + \sup_{q' \in \subdiff F(x)} \iprod{q'-q}{z-x}
    \]
    Because $q \in \interior \subdiff F(x)$, there exists $\epsilon>0$ such that $\B(q, \epsilon) \subset \subdiff F(x)$. We can therefore take $q'=q+\frac{\gamma}{2}(z-x)$ provided $\frac{\gamma}{2}\norm{z-x} \le \epsilon$, that is, if $z \in \B(x, \rho)$ for $\rho=2\epsilon/\gamma$. This immediately yields the claim.
\end{proof}

For the next lemma, we recall we that
$
    \norm{g}_{p,1} \defeq \sum_{i=1}^{n} \norm{g_{i\freevar}}_p,
$
where $g \in \R^{n \times m}$ and we write $g_{i\freevar}=(g_{11},\ldots,g_{1m})$.

\begin{lemma}
    \label{eq:lemma-coupling}
    Let $F(x) \defeq \norm{Wx}_{p,1}$ for some $W \in \R^{(n \times m) \times n}$. Assume for all $i=1,\ldots,n$ the existence of $k_i \in \{1,\ldots,n\}$ such that $[Wx]_{k_i\freevar} = 0$ and $W_{k_i\freevar,i} \ne 0$.
    Then $\interior \subdiff F(x) \ne \emptyset$.
\end{lemma}

\begin{proof}
    We have $\subdiff F(x)=W^* \subdiff \norm{\freevar}_{p,1}(Wx)$, where $\subdiff \norm{\freevar}_{p,1}(g)=\prod_{i=1}^n \subdiff \norm{\freevar}_p(g_{i\freevar})$.
    From our assumptions, for all $i=1,\ldots,n$ we have $\subdiff \norm{\freevar}_p([Wx]_{k_i\freevar})=\B_{p^*}$ for the dual unit ball $B_{p^*} \defeq \{ q \in \R^m \mid \norm{q}_{p^*} \le 1\}$ with $1/p+1/p^*=1$.
    Hence, for all $i=1,\ldots,n$, the projection of $\subdiff F(x)$ to the $i$:th coordinate,
    \begin{equation*}
        \begin{split}
        [\subdiff F(x)]_i & = [W^* \subdiff \norm{\freevar}_{p,1}(Wx)]_i
        =\sum_{k=1}^n \iprod{W_{k\freevar,i}}{[\subdiff \norm{\freevar}_{p,1}(Wx)]_{k\freevar}}
        \\
        &
        =\sum_{k \ne k_i} \iprod{W_{k\freevar,i}}{[\subdiff \norm{\freevar}_{p,1}(Wx)]_{k\freevar}}
        +\iprod{W_{k_i\freevar,i}}{\B_{p^*}}.
        \end{split}
    \end{equation*}
    The last term has non-empty interior. Hence $\interior [\subdiff F(x)]_i \ne \emptyset$ for all $i=1,\ldots,n$. Since $\interior\subdiff F(x) \supset \prod_{k=1}^n \interior [\subdiff F(x)]_i$, the claim follows.
\end{proof}

The next theorem shows that forward-differences discretised total variation is locally strongly convex around a ``strictly piecewise constant'' $\realoptx$.
%If, in fact, if the solution pair $(\realoptx, \realopt q)$ for $\realopt q \defeq -\grad G(\realoptx)$ is ``strictly complementary'' in the sense that $\realopt q \in \interior \subdiff F(\realoptx)$ (instead of merely  $\realopt q \in \subdiff F(\realoptx)$, which holds by definition), then by \cref{theorem:beta=0}, we can expect the iterates of \cref{alg:gn-overrelax} to converge to $\realoptx$ provided we initialize close enough to this point.

\begin{theorem}
    \label{thm:tv-strictly-piecewise-constant}
    Let $F(x)=\norm{\grad_hx}_{p,1}$ for $\grad_h \in \R^{(n_1 \times n_2 \times 2) \times (n_1 \times n_2)}$ the forward differences operator with (discrete) Neumann boundary conditions and cell width $h>0$.
    Assume that $\realoptx \in {n_1 \times n_2}$ is \emph{strictly piecewise constant} in the sense that for each pixel coordinate $(i,j) \in \{1,\ldots,n_1\} \times \{1,\ldots,n_2\}$ there exists a neighboring pixel coordinate
    \[
        (k_{ij},k_{ij}) \in \mathcal{N}_{i,j} \defeq  \{1,\ldots,n_1\} \times \{1,\ldots,n_2\} \isect \{(i,j),(i+1,j),(i,j+1),(i-1,j),(i,j-1)\}
    \]
    with $[\grad_h \realoptx]_{k_{ij}k_{ij}\freevar}=0$.
    Then $\interior \subdiff F(\realoptx) \ne \emptyset$.
    In particular, for any $\gamma>0$ and $\realopt q \in \interior \subdiff F(\realoptx)$ and $\rho>0$ such that $F(z)-F(\realoptx) \ge \iprod{\realopt q}{z-\realoptx} + \frac{\gamma}{2}\norm{z-\realoptx}^2$ for all $z \in \B(\realoptx, \rho)$.
\end{theorem}

\begin{proof}
    The strict piecewise constancy assumption verifies with $W=\grad_h$ for all $i=1,\ldots,n_1$ and $j=1,\ldots,n_2$ the existence of $(k, k) = (k_{ij},k_{ij}) \in \{1,\ldots,n_1\} \times \{1,\ldots,n_2\}$ such that $[W\realoptx]_{kk\freevar} = 0$ and $W_{kk\freevar,ij} \ne 0$. The non-empty interior of the subdifferential is now a consequence of \cref{eq:lemma-coupling}.
    The strong convexity at $\realoptx$ then follows from \cref{lemma:interior-local-sc}.
\end{proof}

If the solution is not strictly piecewise constant at some pixel, then the fitting term $G$ has to provide the corresponding second-order growth. This is reasonable to expect, as total variation whenever allowed by the fitting term, would produce piecewise constant solutions.

\begin{corollary}
    Let $F(x)=\norm{\grad_hx}_{p,1}$ for $\grad_h \in \R^{(n_1 \times n_2 \times 2) \times (n_1 \times n_2)}$ the forward differences operator with (discrete) Neumann boundary conditions. Let $\realoptx \in \inv{[\subdiff_C J]}(0)$ be a Clarke-critical point of $J$.
    For all pixels $(i,j) \in \{1,\ldots,n_1\} \times \{1,\ldots,n_2\}$ such that $-[\grad G(\realoptx)]_{ij} \not \in \interior [\subdiff F(\realoptx)]_{ij}$ (in particular, if $(i, j)$ fails the strict piecewise constancy assumption of \cref{thm:tv-strictly-piecewise-constant} in the sense that there exists no $(k_{ij},k_{ij}) \in \mathcal{N}_{i,j}$ with $[\grad_h \realoptx]_{k_{ij}k_{ij}\freevar}=0$), assume that $[\grad A(\realoptx)\grad A(\realoptx)^*]_{ij,ij} \ge 2\rho+\theta$ for some $\theta>0$.
    Take any $\beta \in (0, \theta)$ satisfying \eqref{eq:gn-convergence:wbound} and initialize $z^0$ close enough to $\realoptx$.
    In \cref{alg:gn-overrelax} use \eqref{eq:beta=0:alg-mod}.
    Then the conclusions of \cref{thm:gn-convergence} hold.
\end{corollary}

\begin{proof}

    With $\realopt{q} \defeq - \grad G(\realoptx)$ let $\mathcal{S}$ be the set of pixel coordinates $(i, j)$ satisfy  $\realopt{q}_{ij} \in \interior [\subdiff F(\realoptx)]_{ij}$.
    Then, if $(i, j) \not\in \mathcal{S}$, we have  $-[\grad G(\realoptx)]_{ij} \in \mathop{\mathrm{bd}} [\subdiff F(\realoptx)]_{ij}$.
    We take $\gamma=2\rho+\theta$ and $\Gamma$ such that $[\Gamma]_{ij,ij}=\gamma$ for pixels $(i,j) \in \mathcal{S}$ and zero in all other entries.
    Then, proceeding as in \cref{lemma:interior-local-sc}, we deduce the existence of $\rho>0$ such that
    \[
        F(z)-F(\realoptx) \ge \iprod{\realopt q}{z-\realoptx} + \frac{1}{2}\norm{z-\realoptx}_\Gamma^2
        \quad (z \in \B(\realoptx, \rho)).
    \]
    By our assumptions we also have $[\grad A(\realoptx)\grad A(\realoptx)^*]_{ij,ij} \ge \Gamma = \gamma \Id =(2\rho+\theta) \Id$.
    The rest follows from \cref{theorem:beta=0}.
\end{proof}

%!TEX root=./gn_overrelax.tex
%%%%%%%%%%%%%%%%%%%%%%%%%%%%%%%%%%%%%%%%%%%%%%%
\section{Additional cases (7--12)}
\label{sec:morecases}
%%%%%%%%%%%%%%%%%%%%%%%%%%%%%%%%%%%%%%%%%%%%%%%

Case 7 is complementary to Case 2; it uses the same geometry and same regularization scheme (Scheme 3) but true conductivity is different. In this case, the target contains a square-shaped inclusion with conductivity of $10^{-3}$ S/m and a conductive circular inclusion with conductivity $0.28$ S/m. The conductivity of the constant background is 0.028 S/m.
\crefname{figure}{Figures}{Figures}
The results of Case 3 are shown in \cref{fig:WT_Sim3_STV_J}--\ref{fig:WT_Sim3_STV_Recos}, and \cref{table:WTSim3STV}.
\crefname{figure}{Figure}{Figures} 

\cref{fig:WT_Sim3_STV_J} shows that RIPGN with relaxation parameters $w=1$ and $w=9/10$ does not converge. Furthermore, the relative error is considerably higher as the total variation regularization tends to round the shape of the resistive inclusion \cite{gonzalez2017isotropic}. In addition, the range of the conductivity is flattened. It is also notable that the fit in this case is better in terms of the objective function than in Case 2. Interpolating the true conductivity into the inversion mesh gives $J(\sigma_\mathrm{true}) = 1.0918\cdot 10^5$ and RE $= 3.8874$ \%.
\input{"Figures/Layout/Case7".tex}
\input{"Tables/WT_Sim3_STV".tex}

Similarly to Case 7, Case 8 is complementary to Case 3. In this case, the comparison to NL-PDPS is omitted due to excessively long computational times of NL-PDPS. The results of Case 8 are shown in \cref{fig:WT_Sim3_TV_J}--\ref{fig:WT_Sim3_TV_Recos}, and \cref{table:WTSim3TV}. Again, the computational times and the relative errors are improved when compared to the smoothed TV solutions in Case 3 (cf. \cref{table:WTSim3STV}), similarly to what happened between Cases 2 and 4. Also, the differences in computational times and relative errors between Case 4 and 5 are analogous to differences between Case 2 and 3.
\input{"Figures/Layout/Case8".tex}
\input{"Tables/WT_Sim3_TV".tex}    

\subsubsection{Cases 9 \& 10: Water tank experiments}
\label{sssec:WTE}
In Cases 9--10, we evaluate RIPGN with experimental data, using a water tank, the geometry of which corresponds to Cases 1--3 (and 7--8). The same objective function (Scheme 3; \cref{sssec:SNSR}) and parameters chosen in Cases 3 and 8 are used in these reconstructions. All reconstructions are computed with relaxation parameter $w=3/4$. 

Reconstructions in Cases 9--10 are shown in \cref{fig:WT_Meas_recos}. In both cases, the plastic inclusions appear as areas of low conductivity, and in Case 10, the metal inclusion causes an area of increased conductivity. These areas are able to capture the locations of the inclusions well and are easily distinguished from the background as the conductivities of the background and these areas are flat and sharp-edged. The background conductivity in both cases is between 0.02 S/m and 0.03 S/m, which is in the range of typical drinking water in room temperatures, and as expected, the conductivity near the plastic inclusion is very low compared to the background. However, there is some contrast loss in the conductivity around the metal inclusion in Case 10; the conductivity in this region is only about twice as much as the background (see \cref{ssec:effectsstv}). Furthermore, in both cases, the shapes of the inclusions are slightly distorted. This kind of distortion can be caused by a small discrepancy between the geometry of the mesh and the actual measurement setup and other modeling errors. The roundness of the objects could reinforced by, for example, increasing the value of the regularization parameter $\alpha$, but the parameter selection for the regularization is beyond the scope of this paper.

The results of the water tank experiments (Cases 9-10) confirm that the RIPGN method proposed in this paper is applicable to EIT imaging also with real measurement
data.
\begin{figure}[!tbp]
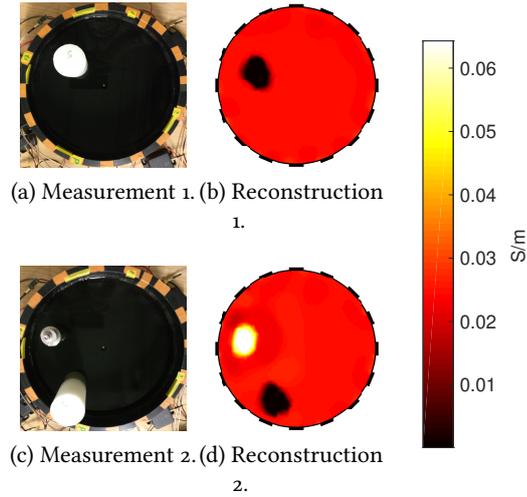

    \centering
    \begin{minipage}{0.45\textwidth}  
        \centering
        \subfloatrecoi{Measurement 1.}{0.33\textwidth}{"Figures/Photos/WT_Meas1_True".jpg}
        \subfloatrecoi{Reconstruction 1.}{0.33\textwidth}{"Figures/WT_Meas/WT_Meas1_w0.75".pdf}\\
        \subfloatrecoi{Measurement 2.}{0.33\textwidth}{"Figures/Photos/WT_Meas2_True".jpg}
        \subfloatrecoi{Reconstruction 2.}{0.33\textwidth}{"Figures/WT_Meas/WT_Meas2_w0.75".pdf}     
    \end{minipage}
    \subfloatcolorbar{-0.5cm}{-2.75cm}{5.5cm}{"Figures/WT_Meas/WT_Meas1_Colorbar".pdf}%
    \caption{Case 9 (top row) and Case 10 (bottom row). Photos of the measurement setup (left column) and the TV-based RIPGN-reconstructions with $w=3/4$.}
    \label{fig:WT_Meas_recos}
\end{figure}\FloatBarrierA% 

\subsubsection{Cases 11 \& 12: Sensing skin experiments}

Case 11 is complementary to Case 4; the measurements are done using the same sensing skin setup and computations use the same scheme (Scheme 3). An additional crack was made on the sensing for this measurement.
\cref{fig:Skin_Meas2_STV_Recos} (top left) shows a photograph of the sensing skin in Case 11. The results from this dataset are shown in\crefname{figure}{Figures}{Figures}  \cref{fig:Skin_Meas2_STV_J}--\ref{fig:Skin_Meas2_STV_Recos} and in \cref{table:SkinMeas2STV}.\crefname{figure}{Figure}{Figures}
In this case, RIPGN with relaxation parameter $w=1/4$ converges better than with the other relaxation parameters, including $w=1/2$. Although the convergence is better with $w=1/4$, \cref{fig:Skin_Meas2_STV_Recos} shows that impact of the relaxation parameter on the reconstruction quality is still fairly negligible. Contrarily, \cref{fig:Skin_Meas2_STV_J} and \cref{table:SkinMeas2STV} show that, again, the relaxation parameter heavily affects the computation times.

\input{"Figures/Layout/Case11".tex}
\input{"Tables/Skin_Meas2_STV".tex}

Case 12 is complementary to Case 5; it uses Scheme 3 and the same measurements as in Case 11. Results in Case 11, in are shown in \cref{fig:Skin_Meas2_TV_J}--\ref{fig:Skin_Meas2_TV_Recos}\crefname{figure}{Figure}{Figures} and \cref{table:SkinMeas2TV}. Differences between Case 11 and Case 12 are fairly similar to differences between Case 4 and 5.
\input{"Figures/Layout/Case12".tex}
\input{"Tables/Skin_Meas2_TV".tex} 

%\FloatBarrierA
%%%%%%%%%%%%%%%%%%%%%%%%%%%%%%%%%%%%%%%%%%%%%%%
\section{Additional reconstructions in Cases 1--6}
\label{sec:morecasestwo}
%%%%%%%%%%%%%%%%%%%%%%%%%%%%%%%%%%%%%%%%%%%%%%%

\crefname{figure}{Figures}{Figures}
\cref{fig:WT_Sim1_Recosall}--\ref{fig:WT3D_Meas1_View1_Recosall} show all reconstruction images computed in Cases 1--6, respectively.
\crefname{figure}{Figure}{Figures} 
\input{"Figures/Layout/Case1_all".tex}
\input{"Figures/Layout/Case2_all".tex}
\input{"Figures/Layout/Case3_all".tex}
\input{"Figures/Layout/Case4_all".tex}
\input{"Figures/Layout/Case5_all".tex}
\input{"Figures/Layout/Case6_all".tex}
\FloatBarrierA
%%%%%%%%%%%%%%%%%%%%%%%%%%%%%%%%%%%%%%%%%%%%%%%
\section{Complementary proximal mappings}\label{app:alg_details}
%%%%%%%%%%%%%%%%%%%%%%%%%%%%%%%%%%%%%%%%%%%%%%%

\cref{tab:proximalsG} collects the proximal mappings required in the algorithm implementations.
\begin{table}[!htbp]
    \centering    
        \caption{Proximal mappings of $G$ utilized in the algorithm implementations. For the hypercube $\zset$, $\protect\mathrm{proj}_{\left[ \zset_\mathrm{min},\zset_\mathrm{max} \right] }(x_i) = \mathrm{max} \left( \mathrm{min} \left(x_i, \zset_\mathrm{max} \right), \zset_\mathrm{min}   \right)  $.}   
    \begin{tabular}{ l | l }
        \hline
        $G(x)$  & i'th component of $\prox{t}{G}{x}$ \\
        \hline
        \hline
        \proxGtabaa & \proxGtabac \\
        \proxGtabba & \proxGtabbc \\
        \proxGtabca & \proxGtabcc \\
        \proxGtabda & \proxGtabdc \\
    \end{tabular}
    \label{tab:proximalsG}
\end{table}
\FloatBarrierA

%%%%%%%%%%%%%%%%%%%%%%%%%%%%%%%%%%%%%%%%%%%%%%%
\section{Details on the computation of the Jacobian}\label{app:jacobian_details}
%%%%%%%%%%%%%%%%%%%%%%%%%%%%%%%%%%%%%%%%%%%%%%%
The Jacobian $\nabla I(\sigma^k)^*$ can be constructed from the partial derivatives we described briefly in \cref{ssec:FME},
\begin{equation}\label{eq:dIds2}
	\frac{\partial I^p(\sigma^k)}{\partial \sigma_i} = -{\mathcal{K}}D(\sigma^k)^{-1}  \frac{\partial D(\sigma^k)}{\partial \sigma_i}\theta^p(\sigma^k),
\end{equation}
where we highlighted that $I^p = I^p(\sigma_1,\ldots,\sigma_N)$, $D=D(\sigma_1,\ldots,\sigma_N)$, and $\theta^p=\theta^p(\sigma_1,\ldots,\sigma_N)$ depend on the iteration through $\sigma^k$.
Note that $\theta^p$ can be obtained by solving the forward problem \eqref{eq:CEMfwd}. Furthermore, $x^* \defeq {\mathcal{K}}D(\sigma^k)^{-1}$, can be solved from the linear system $D(\sigma^k)^*x = \mathcal{K}^*$, similarly to the forward problem. 
Also note that $D$, as defined in \eqref{eq:CEMfwd}, is linear in $\sigma$, so the term $\frac{\partial D(\sigma^k)}{\partial \sigma_i}$ is independent of the iteration and can be preconstructed.
%This part can be computed with the same algorithm that is used to compute $D_1$, by setting $i$'th component of $\sigma$ to $1$ and others to $0$.
Note that the matrix $\frac{\partial D(\sigma^k)}{\partial \sigma_i} \in \R^{(N+L-1) \times (N+L-1)}$, is very sparse: if the degree of the node $i$ in the FEM grid is $n$, then this matrix has maximum of $(n+1)^2$ non-zero elements. Instead of storing it as a compressed sparse column (CSC) matrix, we store it as a dense $(n+1) \times (n+1)$ matrix together with indexing arrays to extract the relevant components of ${\mathcal{K}}D(\sigma^k)^{-1}$ and $\theta^p(\sigma^k)$ to compute the product in \eqref{eq:dIds2}.
This product can be computed very cheaply and fully in parallel over the nodes $i$.
Due to the substantial reduction in indexing, this approach in practise significantly improves the computational time of the Jacobian compared to CSC matrices.
Finally, the $i$:th column of Jacobian matrix is $\frac{\partial I}{\partial \sigma_i} = (\frac{\partial I^1}{\partial \sigma_i},\ldots,\frac{\partial I^L}{\partial \sigma_i})$.

\bibliographystyle{jnsao}
%\bibliography{abbrevs,gn_overrelax}

\end{document}